\documentclass[11pt]{amsart}

\usepackage{mathrsfs}
\usepackage{amsfonts,amssymb,amsmath,amsthm}
\usepackage{latexsym}
\usepackage{bm}
\usepackage{stmaryrd}
\usepackage{geometry}
\usepackage{indentfirst}
\usepackage{titletoc}
\usepackage{appendix}
\usepackage{enumitem}
\usepackage{graphicx}
\usepackage{float}
\usepackage{booktabs}
\usepackage{longtable}
\usepackage{tikz}
\usetikzlibrary{arrows.meta,positioning}
\usepackage{xcolor}

\usepackage{cite}
\usepackage[colorlinks=true,citecolor=red]{hyperref}

\allowdisplaybreaks[4]
\newcommand{\blackhyperref}[2]{%
  {\hypersetup{linkcolor=black}\hyperref[#1]{#2}}%
}

\newcommand{\R}{\mathbb R}

\newcommand{\D}{\mathbb D}
\newcommand{\Sph}{\mathbb S^2}
\newcommand{\cL}{\mathcal L}
\newcommand{\cA}{\mathcal A}
\newcommand{\dd}{\,\mathrm d}

\newcommand{\out}{\mathrm{out}}

\numberwithin{equation}{section}
\newtheorem{theorem}{Theorem}[section]
\newtheorem{lemma}[theorem]{Lemma}
\newtheorem{corollary}[theorem]{Corollary}
\newtheorem{proposition}[theorem]{Proposition}
\newtheorem{definition}[theorem]{Definition}
\theoremstyle{remark}
\newtheorem{remark}[theorem]{Remark}

\begin{document}

\title[Controlled heat flow and blow up]{Controlled harmonic map heat flow and blow-up\\
from the disk to the Ssphere}

\author{Shengquan Xiang}
\address{School of Mathematical Sciences, Peking University,
Beijing 100871, P. R. China}
\email{shengquan.xiang@math.pku.edu.cn}

\begin{abstract}
We study the one-corotational harmonic map heat flow from the disk to the sphere with a Dirichlet boundary control.  We introduce a control-adapted gluing strategy to prevent blow-up. More precisely, we construct explicit upper profiles and well-prepared boundary traces such that every nonnegative datum below one of these profiles and having boundary value above $\pi$ blows up under the corresponding constant trace, whereas the corresponding solution is global under every control below the associated well-prepared boundary trace. Controls with uniformly bounded time derivative also prevent infinite-time concentration.  
We also show that boundary control can drive any nonnegative datum to blow-up before any prescribed time, whereas bounded controls cannot prevent suitably fast concentration.
\end{abstract}

\subjclass[2020]{35K58, 35B44, 58J35, 93C20}
\thanks{\textit{Keywords.} Harmonic map heat flow, boundary control, CDY profile, finite-time blow-up, infinite-time concentration, avoidance of blow-up,  control-adapted gluing strategy.}

\maketitle

\setcounter{tocdepth}{1}
\tableofcontents

\section{Introduction}\label{sec:introduction}

The harmonic map heat flow is the $L^2$-gradient flow of the Dirichlet energy.  The foundational existence and regularity theory was developed in \cite{Eells-Sampson-1964,Struwe-1985,Chen-Struwe-1989}; see the monograph of Lin and Wang \cite{Lin-Wang-2008} for an elegant exposition.  The first finite-time blow-up example was constructed in dimensions higher than two by Coron and Ghidaglia \cite{Coron-Ghidaglia-1989}.  Coron subsequently constructed initial and Dirichlet boundary data on $B^3$ for which the flow admits infinitely many weak solutions \cite{Coron-1990}.

The singularity theory was further developed through the blow-up and global-existence results of Chen and Ding \cite{Chen-Ding-1990} and the analysis of bubbling at singular times by Qing and Qing--Tian \cite{Qing-1995,Qing-Tian-1997}.  The energy identity at finite singular times was proved by Lin and Wang \cite{Lin-Wang-1998}.  For the related gradient estimates and blow-up analysis of stationary harmonic maps, we refer to Lin \cite{Lin-1999}.  Global solutions may also lose compactness through bubbling along sequences of times tending to infinity; see Topping \cite{Topping-2004}.  For maps from a disk to the sphere, Chang and Ding proved a global-existence result \cite{Chang-Ding-1991}, while Chang, Ding, and Ye discovered the sharp radial boundary threshold for finite-time blow-up \cite{Chang-Ding-Ye-1992}.
\vspace{2mm}

Recently, control problems for geometric evolution equations have been investigated in a variety of settings.  For the harmonic map heat flow, we refer to the external-field control in \cite{Liu-2020} and to small-time global controllability results between harmonic maps in \cite{Coron-Xiang-2025}; see also results on finite modes control \cite{Gussetti-2026} and on related Landau-Lifshitz-Gilbert equations \cite{Jiao-Tai-2026, Mukherjee-Fahim-Hausenblas-2026}.  
Controllability and stabilization of wave maps from a circle to a sphere were developed in \cite{Krieger-Xiang-2024,Coron-Krieger-Xiang-2025}, and the global relation between controllability and homotopy for general compact targets was established in \cite{Coron-Krieger-Xiang-arXiv-2025}. 

\smallskip
\smallskip

\begin{center}
\itshape Can boundary controls create or prevent blow-up in the harmonic map heat flow,\\ and if so, how?
\end{center}
\smallskip

We partially answer both parts of this question in the one-corotational setting. For the prevention of blow-up, we introduce a control-adapted gluing strategy as the main new ingredient.  We construct explicit upper profiles for which every nonnegative datum below the profile with boundary value above $\pi$ blows up under its constant trace, whereas every compatible trace in a $C^1$ family gives a global solution; controls with uniformly bounded time derivative also prevent infinite-time concentration.

For a different perspective on blow-up and control, we refer to Lin and Zaag \cite{Lin-Zaag-2022,Lin-Zaag-2025}, who study feedback control of blow-up points for scalar heat equations with internal control. See also Le Balc'h and Souplet \cite{Le-Balch-Souplet-2026} for internal global and regional blow-up controllability of weakly superlinear heat equation.
\vspace{2mm}

Let $\D:=\{x\in\R^2:|x|<1\}$.  For a one-corotational map $u:[0,T)\times\D\to\Sph$, we write
\begin{equation*}
u(t,r,\theta) =\bigl(\sin q(t,r)\cos\theta,\sin q(t,r)\sin\theta, \cos q(t,r)\bigr).
\end{equation*}
The harmonic map heat flow reduces to
\begin{equation}\label{eq:flow:radial}
\begin{cases}
q_t=q_{rr}+\dfrac1r q_r-\dfrac{\sin(2q)}{2r^2}, &(t,r)\in\mathbb R^+\times(0,1),\\
q(t,0)=0,\quad q(t,1)=g(t),&t\in\mathbb R^+,\\
q(0,r)=q_0(r),&r\in[0,1].
\end{cases}
\end{equation}
Here $g$ is the {\it Dirichlet boundary control}. The main part of the equation can be written as
\begin{equation}\label{eq:op:L}
q_t=\cL q  \textrm{ \; with \;}  \cL q:=q_{rr}+\frac1r q_r-\frac{\sin(2q)}{2r^2}.
\end{equation} 

We work with the one-corotational data space
\begin{equation}
X_{\mathrm{cor}}:=\left\{q\in C([0,1]): \; q(0)=0,\quad q(r)/r\ \text{extends continuously to }[0,1]\right\}.
\end{equation}
Here and below, a solution $q$ is called {\it classical} on $(0,T_{\max})$ if, for every $0<\tau<S<T_{\max}$, one has $q\in C^{1,2}\bigl([\tau,S]\times[0,1]\bigr)$,
and $q$ satisfies \eqref{eq:flow:radial} pointwise for $(t,r)\in[\tau,S]\times(0,1)$, together with the boundary conditions at $r=0,1$, and that $q$ attains the initial datum $q_0$ in $X_{\mathrm{cor}}$.
\smallskip
\smallskip

For such a solution, {\it finite-time blow-up} is characterized by
\begin{equation*}
T_{\max}<+\infty,
\qquad
\limsup_{t\to T_{\max}}
\|q_r(t,\cdot)\|_{L^\infty(0,1)}=+\infty.
\end{equation*}
A global classical solution develops {\it infinite-time concentration} if
\begin{equation*}
T_{\max}=+\infty,
\qquad
\limsup_{t\to+\infty}
\|q_r(t,\cdot)\|_{L^\infty(0,1)}=+\infty.
\end{equation*}

The solution framework used in this paper is formulated for initial data and boundary controls of lower regularity.
We therefore define the corresponding notions on well-posedness and blow-up for maximal mild solutions in Appendix~\ref{subsec:wp-mild}.  For classical solutions, these definitions agree with the usual notions above.
\smallskip

We first record two fundamental results; see \cite[Section~6.3]{Lin-Wang-2008}.
\begin{theorem} \label{thm:in:threshold}
Let $q_0\in X_{\mathrm{cor}}$ be a smooth one-corotational datum with $q_0(1)=b$, and consider the maximal solution of the constant-boundary problem associated with \eqref{eq:flow:radial}. 
\smallskip

\begin{enumerate}[label=\textup{(\roman*)},leftmargin=2.4em]
\item (Chang--Ding--Ye~\cite{Chang-Ding-Ye-1992}) If $b>\pi$, the solution blows up in finite time.
\smallskip

\item (Chang--Ding~\cite{Chang-Ding-1991} and Grayson--Hamilton~\cite{Grayson-Hamilton-1996}) If
$|q_0(r)|\leq\pi$ for every $r\in[0,1]$, the solution is global and classical for every positive time.
\end{enumerate}
\end{theorem}

The second conclusion excludes finite-time blow-up, but it does not exclude infinite-time  concentration.  This distinction is part of the motivation for the uniform-in-time conclusion of our main result.
\smallskip

The $k$-equivariant harmonic map heat flow has been extensively studied in the literature. The formation and scale of a vanishing bubble were studied in \cite{Grayson-Hamilton-1996,van-den-Berg-King-Hulshof-2003,Angenent-Hulshof-Matano-2009}, while global-existence and blow-up criteria were developed in \cite{Guan-Gustafson-Tsai-2009}.  Stability, concentration, oscillation, quantized blow-up rates, and higher-equivariance dynamics were subsequently investigated in \cite{Gustafson-Nakanishi-Tsai-2010,Raphael-Schweyer-2013,Raphael-Schweyer-2014,Davila-delPino-Wei-2020,Kim-Merle-2025}.  Bubble decomposition was obtained in the equivariant setting by Jendrej and Lawrie \cite{Jendrej-Lawrie-2023}, and for general flows from $\mathbb R^2$ to $\mathbb S^2$ by Jendrej, Lawrie, and Schlag \cite{Jendrej-Lawrie-Schlag-2025}.

The recent work of Kim and Merle gives a classification of global dynamics in higher equivariance classes \cite{Kim-Merle-2025}.  Along this direction, Kim ruled out bubble trees for finite-energy one-equivariant flows \cite{Kim-2026}.  For time-dependent boundary data, Samuelian obtained a one-bubble description of finite-time singularities and constructed infinite-time blow-up in every equivariance class \cite{Samuelian-2026}, he further excluded finite-time blow-up for $k\geq2$ in \cite{Samuelian-arXiv-2026}.
\smallskip

\subsection{A common datum with opposite boundary-driven dynamics}\label{subsec:opp-dyn}

We first define the {\it stationary harmonic maps},
\begin{equation}\label{eq:in:harmap}
Q_\rho(r):=2\arctan\left(\frac r\rho\right), \quad \forall \rho\in(0,+\infty).
\end{equation}
We also fix $R:=1/10$.  Let $\psi$ be the fixed sign-changing annular mode constructed in Lemma~\ref{lem:annular-mode}, and let $T_1=1/5$.  For $\lambda>0$, we introduce the explicit \emph{upper initial profile},
\begin{equation}\label{eq:in:uppro}
\overline q_{\lambda,0}(r):=
\begin{cases}
Q_\lambda(r),& \forall r\in[0,R],\\
Q_\lambda(r)+3\lambda\psi(r),& \forall r\in(R,1],
\end{cases}
\end{equation}
with the boundary value
\begin{equation}\label{eq:in:upbou}
b_\lambda:=\overline q_{\lambda,0}(1) =Q_\lambda(1)+3\lambda.
\end{equation}

We also introduce the smooth nonnegative boundary trace,
\begin{equation}\label{eq:ma:control}
g_\lambda(t):=Q_\lambda(1)+3\lambda e^{-30t} \textrm{ \; satisfying \;}   g_\lambda(0)=b_\lambda,\;  \|g_\lambda-b_\lambda\|_{L_t^\infty(0,\infty)}\leq 3\lambda.
\end{equation}
The same barrier also controls a family of boundary traces.  Define
\begin{equation}\label{eq:ma:concla}
\mathcal G_\lambda:=\left\{g\in C^1_{\mathrm{loc}}([0,+\infty)): \; 
\begin{aligned}
&0\leq g(t)\leq g_\lambda(t),&& \forall t\in[0,T_1],\\
&0\leq g(t)\leq g(T_1),&& \forall t\in[T_1,+\infty)
\end{aligned}
\right\}.
\end{equation}
\smallskip

We now state the first main result.
\begin{theorem}\label{thm:main}
There exists $\lambda_*>0$ such that, for every $\lambda\in(0,\lambda_*)$, the profile $\overline q_{\lambda,0}$ is positive on $(0,1]$ and $b_\lambda>\pi$.  Let the initial data $q_0\in X_{\mathrm{cor}}$ satisfy
\begin{equation}\label{eq:ma:data}
0\leq q_0(r)\leq\overline q_{\lambda,0}(r) \quad \forall r\in[0,1], \textrm{ \; and } b_0:=q_0(1)\in(\pi,b_\lambda].
\end{equation}

Then the following opposite boundary-driven dynamics hold.
\begin{enumerate}[label=\textup{(\roman*)},leftmargin=2.4em]
\item The maximal mild solution with the constant trace $b_0$ blows up in finite time.
\smallskip

\item For every $g\in\mathcal G_\lambda$ satisfying $g(0)=b_0$, the corresponding maximal mild solution $q_\lambda^g$ is global, thus the solution does not blow-up in finite-time.  At time $T_1=1/5$, it satisfies $0\leq q_\lambda^g(T_1,r)<\pi$ for every $r\in[0,1]$.
If, in addition, one has
\begin{equation}\label{eq:ma:unider}
\sup_{t\geq T_1}|g'(t)|<+\infty,
\end{equation}
then, for every $\tau>0$ there is
\begin{equation}\label{eq:ma:unigra}
\sup_{t\geq\tau}\|(q_\lambda^g)_r(t,\cdot)\|_{L^\infty(0,1)}<+\infty.
\end{equation}
In particular, the solution does not develop infinite-time concentration.
\end{enumerate}
\end{theorem}
An explicit smooth control trace example can be chosen as
\begin{equation}\label{eq:ma:adcon}
g_{\lambda,b_0}(t):=Q_\lambda(1)+\bigl(b_0-Q_\lambda(1)\bigr)e^{-30t}.
\end{equation}
It belongs to $\mathcal G_\lambda$ and produces a solution that is classical for every positive time.

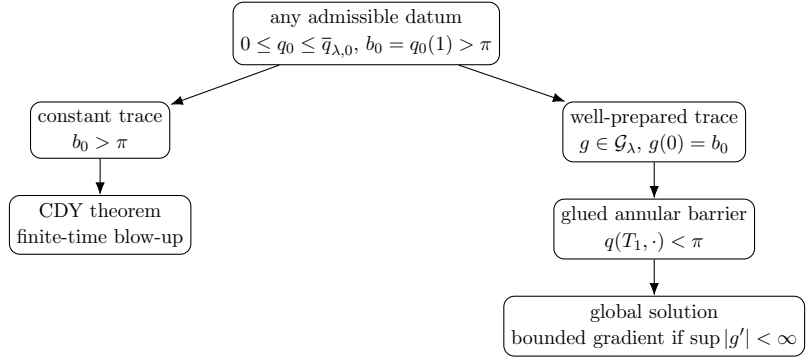
\begin{figure}[H]
\centering
\begin{tikzpicture}[
  scale=.69,
  transform shape,
  >=Latex,
  node distance=7mm and 12mm,
  every node/.style={align=center},
  box/.style={draw,rounded corners,inner sep=4pt}
]
\node[box] (datum) {any admissible datum\\
$0\leq q_0\leq\overline q_{\lambda,0}$, $b_0=q_0(1)>\pi$};
\node[box,below left=of datum] (fixed) {constant trace\\$b_0>\pi$};
\node[box,below right=of datum] (control) {well-prepared trace\\$g\in\mathcal G_\lambda$, $g(0)=b_0$};
\node[box,below=of fixed] (blowup) {CDY theorem\\finite-time blow-up};
\node[box,below=of control] (barrier) {glued annular barrier\\
$q(T_1,\cdot)<\pi$};
\node[box,below=of barrier] (regular) {global  solution\\bounded gradient if $\sup|g'|<\infty$};
\draw[->] (datum) -- (fixed);
\draw[->] (datum) -- (control);
\draw[->] (fixed) -- (blowup);
\draw[->] (control) -- (barrier);
\draw[->] (barrier) -- (regular);
\end{tikzpicture}
\end{figure}

An interesting feature of the result is a boundary-driven phase transition: for this concentrating family, an $O(\lambda)$ perturbation of the boundary trace separates finite-time blow-up from global smoothness. 
Both the upper profile and the boundary perturbation depend on $\lambda$. We do not assert that an arbitrarily small control prevents blow-up for one fixed datum.

\begin{remark} \label{rem:in:mil-cla}
At the minimal regularity level, finite-time blow-up in the preceding theorem refers to the maximal mild solution in Definition~\ref{def:wp:blowup} for control $g\in C_{\mathrm{loc}}^{1}([0,+\infty))$.  If, in addition,  for some $\alpha\in(0,1)$, the control  $g\in C_{\mathrm{loc}}^{1+\alpha/2}((0,+\infty))$,
then Proposition~\ref{prop:wp-reg}\textup{(b)} shows that the mild solution is classical for every positive time before $T_{\max}$.  Remark~\ref{rem:wp:cla-blow} then identifies finite maximal lifespan with  classical gradient blow-up. 
\end{remark}

\medskip

\noindent\textbf{The major difficulty.} 
The initial data in Theorem~\ref{thm:main} already belong to a finite-time blow-up regime under the constant boundary trace, whereas the controlled evolution must remain regular for all time. 
Moreover, the data can highly concentrate near the origin. A fundamental difficulty is that the control acts only at the boundary, whereas blow-up occurs at the origin.  Thus, the control region is spatially separated from the singularity. This is a global control problem highly emphasizing nonlinear behaviors. A direct perturbative linearization argument is not sufficient: concentration takes place at the origin, the radial equation is singular there, and the control acts only at the outer boundary. 

\medskip

\noindent\textbf{The control-adapted gluing strategy.} At a broader level, we introduce a control-adapted gluing strategy that connects boundary control and blow-up. The question is whether the flexibility of the boundary control can alter concentration once the inner and outer dynamics are separated at a fixed interface.  In our construction, the inner region is an exact stationary harmonic map, while the boundary-control flexibility generates an exact nonlinear evolution on the outer annulus.  The two pieces are continuous but have a favorable derivative jump, so their union remains a weak upper solution.  We expect that the same control-adapted gluing   strategy may also be useful for other problems concerning the creation or prevention of blow-up through control. See Section \ref{sec:comparisons} for details.

\subsection{Boundary controls producing arbitrarily fast blow-up}\label{subsec:ind-blow}

Boundary control can also act in the opposite direction.  The next theorem shows that a preparatory positive boundary control can create arbitrarily fast blow-up for any nonnegative initial data.

\begin{theorem} \label{thm:ind-blow}
Let $q_0\in X_{\mathrm{cor}}$ be nonnegative.  For every prescribed $\mathcal T>0$, there exists a smooth nonnegative Dirichlet boundary control $g$, with $g(0)=q_0(1)$, such that the corresponding maximal mild solution of \eqref{eq:flow:radial} blows up at some time $T_*<\mathcal T$.  The solution is classical for every positive time before $T_*$.
\end{theorem}

The proof is based on two observations.  First, a CDY profile starting from a initial scale $\lambda_0$ can blow up after time $O(\sqrt{\lambda_0})$.  Second, in any prescribed positive time, a sufficiently large smooth boundary control raises $q/r$ uniformly and thereby places the solution above such a profile.

Note that sign assumption is essential to the present result.  For an arbitrary sign-changing profile, a positive boundary control does not immediately place the solution in the positive order cone.  Extending Theorem~\ref{thm:ind-blow} to that class would require a new order-preparation lemma or a more flexible lower profile.

\subsection{Unavoidable blow-up under bounded boundary controls} \label{subsec:bou-con}

Theorem~\ref{thm:ind-blow} asserts arbitrarily fast controlled blow-up only in the unbounded control class. There is also a converse obstruction when the size of the boundary control is limited.

\begin{theorem}\label{thm:bou-con}
For every $\varepsilon_*>0$, $M> 0$, and $\mathcal T>0$, there exists a smooth nonnegative one-corotational profile $q_0$, vanishing in a neighborhood of $r=1$, such that $\|q_0\|_{L^\infty(0,1)}\leq\pi+\varepsilon_*$. 
For every $g\in C^1_{\mathrm{loc}}([0,+\infty))$ satisfying $g(0)=0$  and $\|g\|_{L^\infty(0,\mathcal T)}\leq M$, the corresponding maximal mild solution of \eqref{eq:flow:radial} blows up before $\mathcal T$.
\end{theorem}

The proof again uses a collapsing CDY profile, but now only on a fixed inner region $[0,R]$.  We construct an initial datum that dominates this profile on $[0,R]$ and has a strict interface margin at $r=R$. 
A short-time parabolic estimate shows that boundary controls bounded by $M$ cannot destroy this margin before a time independent of $\lambda$.  Choosing the scale sufficiently small then makes the inner CDY core collapse within this time interval and forces finite-time blow-up.

\begin{remark} \label{rem:bou-con:sharp}
The initial-amplitude threshold $\pi$ in Theorem~\ref{thm:bou-con} is sharp in the following sense.  Let $q_0$ be nonnegative and smooth, and suppose that $\|q_0\|_{L^\infty(0,1)}\leq \pi$.
Then Theorem \ref{thm:in:threshold} shows that the solution corresponding to  the compatible constant boundary control  $g\equiv q_0(1)$ is global.
\end{remark}

\vspace{2mm}

\noindent\textbf{Organization of the paper.} Section~\ref{sec:comparisons} explains the control-adapted gluing strategy, the role of the sign-changing mode, and the nonlinear controlled upper barrier.  Section~\ref{sec:annular-mode} selects the sign-changing mode. Section~\ref{sec:ann-bar} constructs the exact annular barrier and proves Theorem~\ref{thm:main}.  Sections~\ref{sec:ind-blow} and \ref{sec:un-blow} prove Theorems \ref{thm:ind-blow}--\ref{thm:bou-con}.  Appendix~\ref{sec:wp-comp} gives the well-posedness, blow-up, and weak comparison framework.  The quantitative CDY subsolution is given in Appendix~\ref{sec:cdy}.
\vspace{2mm}

\section{The control-adapted gluing strategy for preventing blow-up}\label{sec:comparisons}

This section explains the main idea to prove Theorem \ref{thm:main}.  The aim is to build an upper comparison that works for every datum in \eqref{eq:ma:data}.  
We explain the fixed-interface gluing, the sign change forced by the interface and boundary requirements, and the passage from a linear design to an exact nonlinear barrier. 

\medskip

Clearly, a stationary upper barrier is not sufficient. Indeed, every stationary harmonic map $Q_\rho$ in \eqref{eq:in:harmap} stays strictly below $\pi$ on the disk. 

\medskip

\noindent\textbf{The control-adapted gluing strategy.} 
We should not search for one straightforward global exact solution having all the desired properties.  Instead, we split the disk at the fixed radius $R$ and search for the upper barrier in the schematic form
\begin{equation}\label{eq:st:glu}
\begin{aligned}
\overline q_\lambda(t,r) &:=
\begin{cases}
Q_\lambda(r),& \forall r\in[0,R],\\
\overline q_{\lambda,\out}(t,r)= Q_\lambda(r)+\lambda w_\lambda(t,r),& \forall r\in (R,1].
\end{cases}
\end{aligned}
\end{equation}
\begin{itemize}[leftmargin=2.4em]
\item On the inner disk, we keep the \emph{stationary harmonic map} $Q_\lambda$. It is an exact solution. It controls the origin region, and satisfies
$0\leq Q_\lambda(r)/r\leq2/\lambda$.

\item On the fixed outer annulus, we construct an \emph{exact nonlinear evolution} $\overline q_{\lambda,\out}$ satisfying
\begin{equation*}
\begin{cases}
(\overline q_{\lambda,\out})_t =\cL\overline q_{\lambda,\out},&\forall (t,r)\in\mathbb R^+\times(R,1),\\
\overline q_{\lambda,\out}(t,R)=Q_\lambda(R),&\forall t\in\mathbb R^+,\\
\overline q_{\lambda,\out}(t,1)=g_\lambda(t),&\forall t\in\mathbb R^+.
\end{cases}
\end{equation*}
It allows the outer trace to start from the constant blow-up value and then decay while the stationary core $Q_{\lambda}$  remains unchanged. Only the \emph{annular correction} $w_\lambda$ is required to decay, but the profile at time $T_1$  lies strictly below $\pi$.

\item The two exact solutions are continuous at $r=R$, but their radial derivatives are allowed to \emph{jump}.  We choose the sign of this jump so that the interface defect is a nonnegative measure in the supersolution inequality.  The glued function is therefore a \emph{weak upper solution}.
\end{itemize}

Important inner--outer gluing constructions for parabolic concentration in critical heat equations can be found in \cite{delPino-Musso-Wei-2020,delPino-Musso-Wei-2021}.  The idea of gluing was also used by Coron and the author to solve small-time global controllability between stationary harmonic maps  \cite{Coron-Xiang-2025}; see also \cite{Coron-Krieger-Xiang-arXiv-2025}.
The gluing strategy used here has a different role: it joins two exact solutions at a fixed interface and retains the derivative jump as a favorable measure in the weak supersolution inequality.  In this way, the stationary core handles the origin region, while the flexibility of the boundary control handles the exact outer solution.
\medskip

\noindent\textbf{The gluing constraint and the forced sign change.} Continuity at the fixed interface requires $w_\lambda(t,R)=0$.
The radial derivatives of the two exact pieces, however, need not agree. For a continuous and piecewise $C^1$ function $h$, we use the notation $[ \cdot ]_R$ to define the jump
\begin{equation}\label{eq:st:jumcon}
[h_r]_R:=h_r(R+)-h_r(R-).
\end{equation}
Distributional differentiation gives $h_{rr}=(h_{rr})_{\rm reg}+[h_r]_R\delta_R$. Since both pieces in \eqref{eq:st:glu} solve the equation away from $R$, we obtain that at the interface $R$ there is
\begin{equation*}
(\overline q_\lambda)_t-\cL\overline q_\lambda= \left((\overline q_\lambda)_t-\cL\overline q_\lambda\right)_{\rm reg}- [\partial_r\overline q_\lambda]_R\delta_R =-[\partial_r\overline q_\lambda]_R\delta_R.
\end{equation*}
To obtain a weak supersolution after gluing, the derivative jump must have the favorable sign:
\begin{equation}\label{eq:st:favjum}
[\partial_r\overline q_\lambda]_R =\lambda\partial_rw_\lambda(t,R)<0.
\end{equation}
The interface defect is then a nonnegative measure in the supersolution inequality.  
Thus, the jump is not an error to be removed; it is the sign that makes the gluing useful.  The weak formulation and comparison principle are proved in Appendix~\ref{subsec:weak-comp}.

Together with $w_\lambda(t,R)=0$, condition \eqref{eq:st:favjum} forces the correction to be negative immediately to the right of $R$.  At the outer boundary, the opposite sign is needed: the initial trace must satisfy $Q_\lambda(1)+\lambda w_\lambda(0,1)>\pi$, whereas $Q_\lambda(1)=\pi-2\lambda+O(\lambda^3)$.  Thus,  the outer correction $\lambda w_{\lambda}$ needs to change sign on the annulus.

\medskip

\noindent\textbf{The linear sign-changing mode and the approximate barrier.} We next explain, based on approximation, how to find a linear mode.  Substituting the annular expression in \eqref{eq:st:glu} into the exact equation and using $\cL Q_\lambda=0$, we obtain
\begin{equation}\label{eq:st:scaleq}
(w_\lambda)_t=(w_\lambda)_{rr}+\frac1r(w_\lambda)_r -\frac{\sin(2Q_\lambda+2\lambda w_\lambda)-\sin(2Q_\lambda)} {2\lambda r^2}.
\end{equation}
Observe that on the fixed annulus $[R, 1]$, $Q_\lambda$ converges uniformly to $\pi$ as $\lambda\to 0$.  Hence, for uniformly bounded corrections $w_{\lambda}$, the equation converges to the linear evolution
\begin{equation}\label{eq:st:linlim}
(w_{\rm li})_t=\cA w_{\rm li},   \textrm{ \; \; where  \;}  \cA w_{\rm li}:=(w_{\rm li})_{rr} +\frac1r(w_{\rm li})_r-\frac1{r^2}w_{\rm li}.
\end{equation}
This is the reason for working on a fixed annulus: the singular origin region has already been isolated inside the stationary core, while the outer equation has a regular limit for  $\lambda$ small.
\medskip

We choose a separated solution of \eqref{eq:st:linlim},
\begin{equation}\label{eq:st:lindes}
w_*(t,r):=Ae^{-\nu t}\psi(r),
\end{equation}
where $A,\nu>0$ are constants and $\psi$ solves the spectral equation $\cA\psi=-\nu\psi$.  We normalize the scale of $\psi$ by requiring
\begin{equation}\label{eq:st:mode}
\psi''+\frac1r\psi' +\left(\nu-\frac1{r^2}\right)\psi=0, \quad \psi(R)=0  \textrm{ \; and  \;}  \psi(1)=1.
\end{equation}

To turn this separated solution into a suitable approximate annular barrier, we further require
\begin{equation}\label{eq:st:moreq}
\begin{cases}
\psi'(R)<0,\\
\psi\text{ has exactly one zero in }(R,1),\\
A\psi(1)>2,\\
Ae^{-\nu T_1}\displaystyle\max_{R\leq r\leq1}r\psi_+(r)<2 \; \textrm{ \; \; with \; } \psi_+:=\max\{\psi,0\}.
\end{cases}
\end{equation}
\begin{itemize}[leftmargin=2.6em]
\item The first two conditions give the sign structure required by the gluing.  Since $\psi(R)=0$ and $\psi'(R)<0$, the correction is negative immediately to the right of $R$.  The unique sign change then connects this negative part to the positive value $\psi(1)=1$ at the outer boundary. This also gives the favorable derivative jump after gluing.
\smallskip

\item The third condition moves the initial boundary condition, $Q_\lambda(1)+\lambda A \psi(1)$,  above $\pi$.
\smallskip

\item The fourth condition comes from the requirement,
\begin{equation*}
Q_\lambda(r)+\lambda Ae^{-\nu T_1}\psi(r)<\pi,
\quad \forall r\in[R,1].
\end{equation*}
It, together with the perturbative estimates leading to the exact nonlinear evolution, ensures that the controlled solution lies strictly below $\pi$ at time $T_1$.  
\end{itemize}
Section~\ref{sec:annular-mode} selects concrete parameters and verifies all four conditions in \eqref{eq:st:moreq}.  In fact, the last bound is obtained with $1/4$, leaving room to absorb the later nonlinear perturbations.

The piecewise function obtained by gluing $Q_\lambda+\lambda w_*$ to $Q_\lambda$ is {\it only an approximate barrier}.   The next step is to turn this linear design into an exact nonlinear upper solution.

\begin{figure}[H]
\centering
\begin{tikzpicture}[x=1.05cm,y=.9cm,>=Latex]
\draw[->] (0,0) -- (7.3,0) node[right] {$r$};
\draw[->] (0,-1.75) -- (0,1.55) node[above] {$\psi(r)$};
\draw[very thick]
plot[smooth] coordinates {(0,0) (.55,-.80) (1.55,-.70)
  (2.55,-.25) (3.05,0) (4.50,.62) (6.80,1.02)};
\draw[dashed] (3.05,-1.05) -- (3.05,1.10);
\node[anchor=north east] at (-.04,-.04) {$R$};
\node[below] at (6.80,0) {$1$};
\node[align=center] at (1.55,-1.48) {negative part\\
favorable flux jump};
\node[align=center] at (5.35,1.42) {positive part\\
trace above $\pi$};
\draw[->] (.92,-1.24) -- (.50,-.69);
\end{tikzpicture}
\caption{The sign changing forced by the gluing.  
}
\label{fig:sign-mode}
\end{figure}
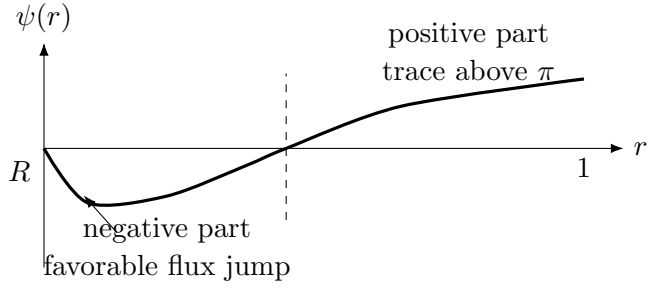

\noindent\textbf{The nonlinear barrier.} Finally, one perturbation turns the approximate barrier into an exact nonlinear one.  The separated solution $w_*$ determines the outer trace.  We impose the same trace on the exact nonlinear annular problem, keeping the initial value $Q_\lambda+\lambda A\psi$ and the inner boundary value $Q_\lambda(R)$. Thus, the boundary control is fixed as
\begin{equation}
g_\lambda(t):=Q_\lambda(1)+\lambda Ae^{-\nu t} =\overline q_{\lambda,\out}(t,1).  \notag
\end{equation}
Straightforward parabolic estimates show that the scaled correction $w_\lambda$ of the exact outer solution converges to $w_*$, both uniformly on the annulus and in the radial derivative at $R$.  

The exact outer solution is then glued to $Q_\lambda$ as in \eqref{eq:st:glu}.  The strict negative derivative jump makes the result a weak supersolution and propagates the upper comparison for every datum satisfying \eqref{eq:ma:data}.  For some well-chosen $T_1$, the nonlinear perturbative estimate gives $\overline q_\lambda(T_1,\cdot)<\pi$ on the whole disk.  The actual controlled solution therefore lies strictly below $\pi$ at time $T_1$.  Finally, stationary upper comparison and parabolic estimates then yield a uniform gradient bound and exclude infinite-time concentration.

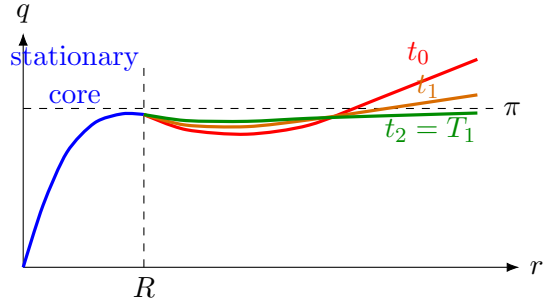
\begin{figure}[H]
\centering
\begin{tikzpicture}[x=1.05cm,y=1.0cm,>=Latex]
\draw[->] (0,0) -- (6.25,0) node[right] {$r$};
\draw[->] (0,0) -- (0,3.10) node[above] {$q$};
\draw[dashed] (0,2.10) -- (5.92,2.10) node[right] {$\pi$};
\draw[dashed] (1.52,0) -- (1.52,2.72);
\node[below] at (1.52,0) {$R$};
\draw[blue,very thick]
plot[smooth] coordinates {(0,0) (.25,.82) (.55,1.52) (.90,1.91) (1.25,2.03) (1.52,2.02)};
\draw[red,very thick]
plot[smooth] coordinates {(1.52,2.02) (2.05,1.82) (2.82,1.76) (3.60,1.88) (4.45,2.22) (5.72,2.75)};
\draw[orange!85!black,very thick]
plot[smooth] coordinates {(1.52,2.02) (2.05,1.88) (2.82,1.86) (3.60,1.95) (4.45,2.08) (5.72,2.28)};
\draw[green!55!black,very thick]
plot[smooth] coordinates {(1.52,2.02) (2.05,1.94) (2.82,1.93) (3.60,1.97) (4.45,2.00) (5.72,2.04)};
\node[blue,align=center] at (.65,2.55) {stationary\\core};
\node[red] at (4.98,2.83) {$t_0$};
\node[orange!85!black] at (5.10,2.39) {$t_1$};
\node[green!55!black] at (5.13,1.82) {$t_2=T_1$};
\end{tikzpicture}
\caption{Annular upper barrier evolution.  The stationary core is kept fixed on $[0,R]$, while the exact nonlinear outer evolution decays under its controlled trace.  At the terminal time $T_1$, the full glued upper barrier lies below $\pi$.}
\label{fig:ann-evo}
\end{figure}

To summarize, the proof of Theorem~\ref{thm:main} follows three steps.
\begin{itemize}[leftmargin=3.0em]
\item[\textit{Step 1.}] Construct the sign-changing mode $\psi$ satisfying outer-boundary and terminal positive-part bounds.  Define the upper initial profile $\overline q_{\lambda,0}$ and approximate upper barrier $Q_{\lambda}+\lambda A e^{-\nu t}\psi$, place its boundary value above $\pi$, and remove $\pi$-crossing at time $T_1$.
\smallskip

\item[\textit{Step 2.}] Use the decaying trace dictated by the linear mode, solve the exact nonlinear annular problem, and glue its solution  $w_{\lambda}$ to the stationary core $Q_{\lambda}$.  The favorable derivative jump gives a weak upper comparison for every datum below $\overline q_{\lambda,0}$.
\smallskip

\item[\textit{Step 3.}] Use the CDY blow-up criterion for the constant trace.  For the controlled trace, propagate the glued upper comparison to time $T_1$, remove the $\pi$-crossing, and dominate the future solution by a stationary harmonic map.
\end{itemize}

Accordingly, Section~\ref{sec:annular-mode} implements Step~1, while Section~\ref{sec:ann-bar} implements Steps~2 and~3.
\section{A sign-changing annular mode}\label{sec:annular-mode}

In this section, we construct a function $\psi$ satisfying the four conditions in \eqref{eq:st:moreq},  as illustrated in Section~\ref{sec:comparisons}.
We first choose $R$ and $\nu$ so that $\psi$ has the required sign change due to Sturm-Liouville theory, and then select $A$ and $T_1$ so that the last two quantitative bounds hold. 

\begin{lemma}\label{lem:annular-mode}
Let us fix the parameters
\begin{equation}\label{eq:mo:para}
R=\frac1{10}, \quad \nu=30, \quad A=3, \; \textrm{ and } \; T_1=\frac15.
\end{equation}
There is a unique function $\psi\in C^\infty([R,1])$ satisfying \eqref{eq:st:mode}, namely,
\begin{equation}\label{eq:mo:intro}
\psi''+\frac1r\psi' +\left(30-\frac1{r^2}\right)\psi=0, \quad \psi(R)=0 \textrm{\; and  \;} \psi(1)=1.
\end{equation}

Moreover, this unique solution satisfies
\begin{gather}
\psi'(R)<0, \quad \psi\text{ has exactly one zero in }(R,1),  \label{eq:mo:sign} \\
M_\psi:=\max_{R\leq r\leq1}r\psi_+(r)< 23  \textrm{\; with \;} \psi_+:=\max\{\psi,0\}, \textrm{ and, }\label{eq:mo:pos} \\
A\psi(1)=3>2  \textrm{ \; and  \; }  Ae^{-\nu T_1}M_\psi<\frac14. \label{eq:mo:quan}
\end{gather}
\end{lemma}

\begin{remark} \label{rem:mo:flex}
The proof below gives more than the particular choice $R=1/10$: for every $R\in(0,1/4)$, the choices $A=3$, $\nu=30$, and $T_1=1/5$ satisfy all the estimates in Lemma~\ref{lem:annular-mode}, with the solution of \eqref{eq:mo:intro}.  More generally, any suitable $R,\nu,A,T_1$ satisfying the four strict conditions in \eqref{eq:st:moreq} can be used to fix $\psi$ and construct the annular barrier.  The constants above are chosen for simplicity and are not optimized.
\end{remark}

\begin{proof}
We divide the proof into three steps.  We first place $\nu=30$ between the first two Dirichlet eigenvalues of the annular operator.  We then use Sturm oscillation to obtain \eqref{eq:mo:sign}.  Finally, a harmonic lifting gives the explicit bound \eqref{eq:mo:pos} and allows us to fix $A$ and $T_1$.

\medskip

\noindent\emph{Step 1. Choice of $R$ and $\nu$.} For the moment, let $R\in(0,1/4)$.  Recall that the spectral equation for the separated mode is $-\cA\varphi=\nu\varphi$.  The operator $\cA$ is the radial part of the planar Laplacian in the first angular mode and can be written as
\begin{equation*}
\cA \varphi=\frac1r(r \varphi_r)_r-\frac1{r^2}\varphi.
\end{equation*}
Its natural radial measure is therefore $r\dd r$.  To control the number of zeros of the mode, we consider the positive Dirichlet operator
\begin{equation*}
\mathcal A_R:=-\cA
=-\partial_{rr}-\frac1r\partial_r+\frac1{r^2}
\quad\text{in }L^2((R,1),r\dd r),
\end{equation*}
with boundary conditions $\varphi(R)= \varphi(1)= 0$.  This defines a positive self-adjoint operator with compact resolvent (in abuse of notation, we changed its sign).
Although the normalized mode will satisfy $\varphi(1)=1$, the Dirichlet eigenvalues are precisely the values of $\nu$ for which a solution vanishing at $R$ and $1$.  They help us to understand its oscillation on the annulus  \cite{Zettl-2005}.
\vspace{2mm}

Denote its first two eigenvalues by $\nu_1<\nu_2$.  Its quadratic form is 
\begin{equation*}
Q_R[\varphi]:=  \langle\mathcal A_R \varphi, \varphi\rangle_{L^2((R,1),r \dd r)}= \int_R^1 \left(r|\varphi'(r)|^2+\frac{|\varphi(r)|^2}{r}\right)\dd r, \quad \forall \varphi \in H_0^1(R,1).
\end{equation*}
Notice that the change of unknown $z=\sqrt r\,\varphi$ removes the weight; this change of variable can simplify the variation calculus.  Simple integration by parts gives
\begin{equation}\label{eq:mo:unif}
\int_R^1r|\varphi|^2\dd r=\int_R^1|z|^2\dd r  \textrm{ \; and  \; \;}
Q_R[\varphi] =\int_R^1\left(|z'|^2+\frac{3}{4r^2}|z|^2\right)\dd r.
\end{equation}

We next derive a uniform upper bound for $\nu_1$.  By the Rayleigh--Ritz principle,
\begin{equation*}
\nu_1 =
\inf_{\varphi\in H_0^1(R,1)\setminus\{0\}} \frac{Q_R[\varphi]}{\int_R^1r|\varphi|^2\dd r}.
\end{equation*}
To obtain a bound independent of $R\in(0,1/4)$, we choose a test function supported on the fixed interval $[1/4,1]$.  More precisely, define
\begin{equation*}
\varphi_0(r):=
\begin{cases}
0,& \forall r\in [R,1/4],\\
(r-1/4)(1-r),& \forall r\in (1/4,1].
\end{cases}
\end{equation*}
The function $\varphi_0$ belongs to $H_0^1(R,1)$.  Direct calculations lead to
\begin{equation*}
Q_R[\varphi_0]=\frac{15}{1024}+\frac1{16}\log4, \quad
\int_R^1r|\varphi_0|^2\dd r=\frac{81}{16384}.
\end{equation*}
Consequently, using $\log4<7/5$, we obtain
\begin{equation*}
\nu_1 \leq \frac{Q_R[\varphi_0]}{\int_R^1r|\varphi_0|^2\dd r}
= \frac{80}{27}+\frac{1024}{81}\log4 <21.
\end{equation*}

For the second eigenvalue $\nu_2$, let $\mu_2$ denote the second eigenvalue of the 1d Dirichlet Laplacian $-\partial_{rr}$ on $(R,1)$.  By the change of variable $z=\sqrt r\,u$, \eqref{eq:mo:unif} and the min--max principle,
\begin{align*}
\nu_2 &= \min_{\substack{E\subset H_0^1(R,1)\\ \dim E=2}}   \max_{0\neq z\in E}
\frac{\int_R^1\left(|z'|^2+\frac{3}{4r^2}|z|^2\right)\dd r} {\int_R^1|z|^2\dd r}\\
&\geq \min_{\substack{E\subset H_0^1(R,1)\\ \dim E=2}} \max_{0\neq z\in E}
\frac{\int_R^1|z'|^2\dd r} {\int_R^1|z|^2\dd r}
=\mu_2.
\end{align*}
Since the $k$-th Dirichlet eigenvalues of $-\partial_{rr}$ on $(R,1)$ are $(k\pi/(1-R))^2$, we conclude that
\begin{equation*}
\nu_2\geq\mu_2=\left(\frac{2\pi}{1-R}\right)^2>36.
\end{equation*}

Thus, $\nu=30$ lies strictly between the first two Dirichlet eigenvalues $\nu_1$ and $\nu_2$.

\medskip

\noindent\emph{Step 2. Construct the sign-changing mode.} Let  $h$ be the solution of
\begin{equation}\label{eq:mo:lefnor}
h''+\frac1r h' +\left(30-\frac{1}{r^2}\right)h=0, \qquad h(R)=0, \quad h'(R)=1.
\end{equation}
Since $30$ is not a Dirichlet eigenvalue, $h(1)\neq0$.  The regular Sturm oscillation theorem states that, when $\nu\in(\nu_k,\nu_{k+1})$, the solution with $h(R)=0$ and $h'(R)=1$ has exactly $k$ zeros in $(R,1)$; see, for example, \cite{Zettl-2005}.  Hence, $h$ has exactly one zero in $(R,1)$.  Since $h'(R)=1$, the function is positive immediately to the right of $R$.  
Every zero of a nontrivial solution is simple, so $h$ changes sign at its unique zero and remains negative up to $r=1$.  In particular, $h(1)<0$.
\smallskip

Define $\psi(r):=h(r)/h(1)$.  Then $\psi$ is the unique solution of \eqref{eq:mo:intro}, and $\psi'(R)=1/h(1)<0$.  Thus, $\psi$ is negative immediately to the right of $R$, crosses zero exactly once, and is positive near $r=1$.  This proves the requirement \eqref{eq:mo:sign}.

\medskip

\noindent\emph{Step 3. Choice of $A$ and $T_1$.} Clearly, choosing $A=3$ gives $A\psi(1)=3>2$.  To determine $T_1$, we first bound the constant $M_\psi$ in \eqref{eq:mo:pos}.  Define the harmonic lifting
\begin{equation}\label{eq:mo:harli}
\Phi_R(r):=\frac{r-R^2/r}{1-R^2}.
\end{equation}
It satisfies
\begin{equation*}
-\mathcal A\Phi_R=0, \qquad \Phi_R(R)=0, \quad \Phi_R(1)=1,  \; \textrm{ and }  0\leq\Phi_R\leq1.
\end{equation*}
Set $y:=\psi-\Phi_R$.  Then $y\in H_0^1(R,1)$ and
\begin{equation}\label{eq:mo:lifres}
(\mathcal A_R-30)y=30\Phi_R.
\end{equation}

Let $\{e_n\}_{n\geq1}$ be an orthonormal Dirichlet eigenbasis of $\mathcal A_R$ in $L^2((R,1),r\dd r)$, with corresponding eigenvalues $\{\nu_n\}_{n\geq1}$. 
Now we decompose $y$ and $\Phi$ using $\{e_n\}_{n\geq 1}$,
\begin{equation*}
y=\sum_{n\geq1}y_ne_n \textrm{ \; and \;  } 30\Phi_R=\sum_{n\geq1}f_ne_n.
\end{equation*}
Equation \eqref{eq:mo:lifres} yields $y_n=f_n/(\nu_n-30)$.   Hence, using Parseval's identity, we obtain
\begin{align*}
Q_R[y]= \sum_{n\geq1}\nu_n|y_n|^2
=\sum_{n\geq1}\frac{\nu_n}{(\nu_n-30)^2}|f_n|^2.
\end{align*}
The spectral bounds above on $\nu_1$ and $\nu_2$ imply
\begin{equation*}
\frac{\nu_1}{(30-\nu_1)^2}<1  \textrm{ \; as well as \; } \frac{\nu_n}{(\nu_n-30)^2}<1 \; \; \forall n\geq2.
\end{equation*}
This, together with the fact that $0\leq \Phi_R\leq 1$, gives
\begin{align*}
Q_R[y]\leq \sum_{n\geq1}|f_n|^2= \|30\Phi_R\|_{L^2((R,1),r \dd r)}^2< 450.
\end{align*}

Set $z:=\sqrt r\,y$.  Since $z(R)=z(1)=0$, identity \eqref{eq:mo:unif} and direct estimate yield
\begin{equation*}
|z(r)| \leq\|z'\|_{L^2(R,1)} \leq Q_R[y]^{1/2} <15\sqrt2.
\end{equation*}
It follows that $r|y(r)|=\sqrt r\,|z(r)|<15\sqrt2$. 
Since $\psi=\Phi_R+y$, we conclude that
\begin{equation*}
r\psi_+(r) \leq r\Phi_R(r)+r|y(r)| <1+15\sqrt2, \quad \forall r\in[R,1].
\end{equation*}
This proves the bound in \eqref{eq:mo:pos}.

It remains to verify the last estimate in \eqref{eq:mo:quan} by selecting $T_1$. We choose $T_1:=1/5$, so that $\nu T_1=6$, which gives $e^6>400$. 
Using \eqref{eq:mo:pos}, we obtain
\begin{equation*}
Ae^{-\nu T_1}M_\psi =3e^{-6}M_\psi <\frac14.
\end{equation*}
This proves \eqref{eq:mo:quan} and completes the proof.
\end{proof}

\begin{corollary}
After reducing $\lambda_*$, the positivity assertion in Theorem~\ref{thm:main} follows directly. First, the upper boundary value $b_\lambda$ in \eqref{eq:in:upbou} satisfies
\begin{equation*}
b_\lambda-\pi =3\lambda-2\arctan\lambda >\lambda, \quad \forall \lambda\in(0, 1).
\end{equation*}
Second, the initial upper profile $\overline q_{\lambda,0}(r)$ in \eqref{eq:in:uppro} is positive.  Indeed, $Q_\lambda\to\pi$ uniformly on the fixed annulus as $\lambda\to0^+$ and $\psi$ is bounded, while on the inner disk the profile $Q_\lambda> 0$.   
\end{corollary}

\section{The annular barrier and the proof of Theorem \ref{thm:main}} \label{sec:ann-bar}

This section turns the linear mode and approximate barrier from Section~\ref{sec:annular-mode} into an exact nonlinear upper barrier.  We use the decaying trace and the initial data of the approximate function to solve the exact nonlinear annular problem.  We then glue this outer solution to the stationary core, preserving the favorable interface derivative.  Finally, we propagate the comparison to $T_1$, remove the $\pi$-crossing, and use a stationary upper bound to exclude concentration for all later times. The parameters $(R, \nu, A, T_1)$ are fixed in  Lemma \ref{lem:annular-mode}.

\subsection{The exact nonlinear outer evolution and its trace} \label{subsec:non-ann}
Recall from \eqref{eq:st:linlim}--\eqref{eq:st:lindes} the linear mode on the outer annulus is $w_*(t,r):=3e^{-30t}\psi(r)$.  
 As illustrated in Section \ref{sec:comparisons}, let $\overline q_{\lambda,\out}$ be its unique global mild solution and write $\overline q_{\lambda,\out}=Q_\lambda+\lambda w_\lambda$:
\begin{equation}\label{eq:ba:oupro}
\begin{cases}
(\overline q_{\lambda,\out})_t =\cL\overline q_{\lambda,\out},&\forall (t,r)\in\mathbb R^+\times(R,1),\\
\overline q_{\lambda,\out}(t,R)=Q_\lambda(R),&\forall t\in\mathbb R^+,\\
\overline q_{\lambda,\out}(t,1) =Q_\lambda(1)+3\lambda e^{-30t},&\forall t\in\mathbb R^+,\\
\overline q_{\lambda,\out}(0,r) =Q_\lambda(r)+3\lambda\psi(r),&\forall r\in[R,1].
\end{cases}
\end{equation}
Since $r\geq R$ on the fixed annulus, the nonlinear term $q\mapsto-\sin(2q)/(2r^2)$ is globally Lipschitz, uniformly in $r$.
Standard semilinear Dirichlet theory therefore gives a unique global mild solution of \eqref{eq:ba:oupro}. 
The initial and boundary values are nonnegative for every sufficiently small $\lambda$.
Comparison with zero immediately yields $\overline q_{\lambda,\out}(t,r)\geq0$.
Its outer trace is precisely the control introduced in \eqref{eq:ma:control}.
Thus, $g_\lambda(0)=b_\lambda$, and $g_\lambda$ is nonnegative and strictly decreasing. 
\smallskip

To compare the exact correction $w_\lambda$ with the linear mode $w_*$, we need both uniform convergence on the annulus and convergence of the radial derivative at the interface.  Let $S(t)$ denote the Dirichlet semigroup generated by $\cA$ on $(R,1)$.  The maximum principle and the standard boundary-gradient estimate for uniformly parabolic Dirichlet problems give, for some $C>0$,
\begin{equation}\label{eq:ann-sem}
\|S(t)\varphi\|_{L^\infty(R,1)}\leq\|\varphi\|_{L^\infty(R,1)}, \qquad
|\partial_rS(t)\varphi(R)|\leq Ct^{-1/2}\|\varphi\|_{L^\infty(R,1)},
\end{equation}
for every $\varphi\in L^\infty(R,1)$ and $t\in(0, T_1]$; see, for example, \cite{Lieberman-1996}.

\begin{lemma} \label{lem:non-ann-con}
As $\lambda\to0^+$, one has
\begin{align}
\sup_{0\leq t\leq T_1} \|w_\lambda(t)-w_*(t)\|_{L^\infty(R,1)} &\longrightarrow 0, \label{eq:ba:non-unicon}\\
\sup_{0\leq t\leq T_1} |\partial_rw_\lambda(t,R)-\partial_rw_*(t,R)| &\longrightarrow 0. \label{eq:ba:non-c1con}
\end{align}
\end{lemma}

\begin{proof}
By the definition of $w_\lambda$ and the identity $\cL Q_\lambda=0$, we obtain
\begin{equation}\label{eq:ba:scaleq}
(w_\lambda)_t =\cA w_\lambda +\mathcal R_\lambda(r,w_\lambda),
\end{equation}
where
\begin{equation}\label{eq:ba:non-rem}
\mathcal R_\lambda(r,z) :=\frac z{r^2} -\frac{\sin(2Q_\lambda+2\lambda z)-\sin(2Q_\lambda)} {2\lambda r^2}.
\end{equation}
The initial and boundary values of $w_\lambda$ agree exactly with those of $w_*$.  Since $Q_\lambda\to\pi$ uniformly on $[R,1]$, the mean value theorem gives, for every fixed $M>0$,
\begin{equation}\label{eq:ba:rem-small}
\varepsilon_\lambda(M) :=\sup_{R\leq r\leq1,\ |z|\leq M} |\mathcal R_\lambda(r,z)|= O(\lambda^2) \; \; \textrm{ as } \lambda\rightarrow 0.
\end{equation}
Moreover,
\begin{equation}\label{eq:ba:rem-lip}
|\mathcal R_\lambda(r,z_1)-\mathcal R_\lambda(r,z_2)| \leq\frac2{R^2}|z_1-z_2| \textrm{ \; \; and  \; \;} \mathcal R_\lambda(r,0)=0.
\end{equation}

Set $Z_\lambda:=w_\lambda-w_*$.  It has homogeneous parabolic data and satisfies
\begin{equation}\label{eq:ba:non-duh}
Z_\lambda(t)=\int_0^tS(t-s) \mathcal R_\lambda \bigl(\cdot,w_*(s)+Z_\lambda(s)\bigr)\dd s.
\end{equation}
Since $\mathcal R_\lambda(r,0)=0$, estimate \eqref{eq:ba:rem-lip} gives
$|\mathcal R_\lambda(r,z)|\leq 2|z|/R^2$.
Using \eqref{eq:ba:non-duh}, the contraction estimate in \eqref{eq:ann-sem}, and $\|w_*(t)\|_{L^\infty(R,1)}\leq3\|\psi\|_{L^\infty(R,1)}$, we obtain
\begin{equation*}
\|w_\lambda(t)\|_{L^\infty(R,1)} \leq3\|\psi\|_{L^\infty(R,1)}
+\frac2{R^2}\int_0^t\|w_\lambda(s)\|_{L^\infty(R,1)}\dd s.
\end{equation*}
Hence, Gronwall's inequality gives
\begin{equation*}
\sup_{0\leq t\leq T_1}\|w_\lambda(t)\|_{L^\infty(R,1)}
\leq 3e^{2T_1/R^2}\|\psi\|_{L^\infty(R,1)}=: M_0.
\end{equation*}
In particular, $M_0$ is independent of $\lambda$.  We may therefore use $M=M_0$ in \eqref{eq:ba:rem-small}.  Returning to \eqref{eq:ba:non-duh} and using again the contraction estimate in \eqref{eq:ann-sem}, for $t\in [0, T_1]$, we obtain 
\begin{align*}
\|Z_\lambda(t)\|_{L^\infty(R,1)}
\leq\int_0^t \bigl\|\mathcal R_\lambda(\cdot,w_\lambda(s))\bigr\|_{L^\infty(R,1)} \dd s
\leq t\varepsilon_\lambda(M_0) \leq T_1\varepsilon_\lambda(M_0).
\end{align*}
This gives the proof of \eqref{eq:ba:non-unicon}.
\smallskip

Finally, since $(t-s)^{-1/2}$ is integrable, the second estimate in \eqref{eq:ann-sem} justifies differentiating \eqref{eq:ba:non-duh} at $R$.  Using the definition of $\varepsilon_\lambda(M_0)$, we obtain
\begin{align*}
|\partial_rZ_\lambda(t,R)|  \leq C\int_0^t(t-s)^{-1/2} \bigl\|\mathcal R_\lambda(\cdot,w_\lambda(s))\bigr\|_{L^\infty(R,1)}\dd s
\leq 2C\sqrt{t}\,\varepsilon_\lambda(M_0).
\end{align*}
Therefore, the inequality \eqref{eq:ba:non-c1con} holds:
\begin{equation*}
\sup_{0\leq t\leq T_1} |\partial_rZ_\lambda(t,R)|
\leq2C\sqrt{T_1}\,\varepsilon_\lambda(M_0) \longrightarrow 0.
\end{equation*}
\end{proof}

Finally, to preserve the favorable interface derivative, we set
\begin{equation}\label{eq:ba:intmar}
c_{T_1}:=-3e^{-30T_1}\psi'(R)>0.
\end{equation}
Since $\psi'(R)<0$, for every $t\in[0,T_1]$ we have
\begin{equation}\label{eq:ba:refder}
\partial_rw_*(t,R) =3e^{-30t}\psi'(R) \leq3e^{-30T_1}\psi'(R)= -c_{T_1}.
\end{equation}

Combining \eqref{eq:ba:refder} with \eqref{eq:ba:non-c1con}, we may choose $\lambda_*>0$ sufficiently small so that
\begin{equation*}
\sup_{0\leq t\leq T_1}
|\partial_rw_\lambda(t,R)-\partial_rw_*(t,R)|
<\frac12c_{T_1},
\quad \forall\lambda\in(0,\lambda_*).
\end{equation*}
Consequently,
\begin{equation}\label{eq:ba:negder}
\partial_rw_\lambda(t,R)\leq-\frac12c_{T_1}<0,
\quad \forall (t,\lambda)\in[0,T_1]\times(0,\lambda_*).
\end{equation}

\subsection{The upper barrier and removal of the \texorpdfstring{$\pi$}{pi}-crossing} \label{subsec:ba:rem}
With $\overline q_{\lambda,\out}$ the solution in \eqref{eq:ba:oupro}, 
we define the upper barrier by
\begin{equation}\label{eq:ba:gludef}
\overline q_\lambda(t,r):=
\begin{cases}
Q_\lambda(r),&\forall r\in[0,R],\\
\overline q_{\lambda,\out}(t,r),&\forall r\in[R,1].
\end{cases}
\end{equation}
At $t=0$, this is the upper profile $\overline q_{\lambda,0}$ in \eqref{eq:in:uppro}.  The two pieces agree at $R$, while \eqref{eq:ba:negder} gives
\begin{equation}\label{eq:ba:qujump}
[\partial_r\overline q_\lambda]_R =\lambda\partial_rw_\lambda(t,R)<0, \quad \forall t\in[0,T_1].
\end{equation}

By \eqref{eq:ba:non-unicon}, we may further reduce $\lambda_*$ so that
\begin{equation*}
\|w_\lambda-w_*\|_{L^\infty([0,T_1]\times(R,1))}<\frac14, \quad \forall\lambda\in(0,\lambda_*).
\end{equation*}
Since the positive-part map is Lipschitz and $r\leq1$, combining this estimate with \eqref{eq:mo:quan} gives
\begin{equation}\label{eq:cl:scale}
\max_{R\leq r\leq1}r\bigl(w_\lambda(T_1,r)\bigr)_+ \leq \max_{R\leq r\leq1}r\bigl(w_*(T_1,r)\bigr)_+ +\|w_\lambda-w_*\|_{L^\infty([0,T_1]\times(R,1))} <\frac12.
\end{equation}

On the other hand,
\begin{equation}\label{eq:cl:hargap}
\frac{r\bigl(\pi-Q_\lambda(r)\bigr)}{\lambda}=
\frac{2r}{\lambda}\arctan\left(\frac{\lambda}{r}\right) \longrightarrow 2
\end{equation}
uniformly for $r\in[R,1]$.  After reducing $\lambda_*$ once more, we therefore have
\begin{equation*}
\frac{r\bigl(\pi-Q_\lambda(r)\bigr)}{\lambda}>\frac32, \qquad \forall (r,\lambda)\in[R,1]\times(0,\lambda_*).
\end{equation*}
Combining this estimate with \eqref{eq:cl:scale}, we obtain
\begin{align*}
\overline q_{\lambda,\out}(T_1,r)-\pi &=\lambda w_\lambda(T_1,r)-\bigl(\pi-Q_\lambda(r)\bigr)\\
&\leq\frac{\lambda}{r} \left( r\bigl(w_\lambda(T_1,r)\bigr)_+ -\frac{r\bigl(\pi-Q_\lambda(r)\bigr)}{\lambda} \right)< 0
\end{align*}
for every $r\in[R,1]$.  Thus, $\overline q_{\lambda,\out}(T_1,r)<\pi $ for every  $r\in[R, 1]$, and therefore,  
\begin{equation}\label{in:ovqla:T1}
0\leq \overline q_\lambda(T_1,r)<\pi \quad \forall r\in[0,1].
\end{equation}

\subsection{Weak supersolution gluing and comparison up to time \texorpdfstring{$T_1$}{T_1}}
\label{subsec:weak-glu}

We use the weak subsolution and supersolution definitions for $v=q/r$ and the weak comparison principle from Appendix~\ref{subsec:weak-comp}. 

Set $\overline v_\lambda:=\overline q_\lambda/r$ and $\overline v_{\lambda,\out}:=\overline q_{\lambda,\out}/r$. 
Since the initial and boundary data in \eqref{eq:ba:oupro} are smooth and satisfy the zeroth-order compatibility condition, standard parabolic regularity on $[R, 1]$ gives $\overline v_{\lambda,\out} \in H^1\bigl(0,T_1;L^2(R,1)\bigr) \cap L^2\bigl(0,T_1;H^2(R,1)\bigr)$.
In particular, $\partial_r\overline v_{\lambda,\out}(\cdot,R+)\in L^2(0,T_1)$.
Moreover, $\overline v_{\lambda,\out}$ is classical on $[\tau,T_1]\times[R,1]$ for every $\tau\in(0,T_1)$.
\smallskip

Now, we show that $\overline v_\lambda$ is a weak supersolution, using Lemma \ref{lem:fix-int}. On the inner ball, $Q_\lambda(r)/r$ extends smoothly across the center and is stationary.  Since the two pieces have the same time-independent trace at $R$, they define a function
\begin{equation*}
\overline v_\lambda\in\mathcal W_{T_1}\cap C([0,T_1]\times[0,1]) \textrm{ \; with \; }
\partial_t\overline v_\lambda\in L^2(0,T_1;\mathcal H) \subset L^2(0,T_1;\mathcal V_0').
\end{equation*}
The regular residual $G$ vanishes on both pieces. The inner piece is smooth, while the one-dimensional trace gives both  $R^3\partial_r\overline v_\lambda(\cdot,R-)$ and  $R^3\partial_r\overline v_\lambda(\cdot,R+)$ belong to $L^2(0,T_1)$.
Since $\overline q_\lambda$ is continuous at $R$, \eqref{eq:ba:negder} gives
\begin{equation*}
[\partial_r\overline v_\lambda]_R =\frac1R[\partial_r\overline q_\lambda]_R
=\frac{\lambda}{R}\partial_rw_\lambda(t,R)<0.
\end{equation*}

Lemma~\ref{lem:fix-int} therefore shows that $\overline v_\lambda$ is a weak supersolution on $[0,T_1]$.
\vspace{2mm}

Let $q_0$ satisfy \eqref{eq:ma:data}, and let $g\in\mathcal G_\lambda$ satisfy $g(0)=q_0(1)$.  Denote the corresponding maximal mild solution by $q_\lambda^g$ (by Proposition~\ref{prop:loc-wp}), and set $v_\lambda^g(t,r):= q_\lambda^g(t,r)/r$.
By Lemma~\ref{lem:mild-energy}, $v_\lambda^g$ is both a weak subsolution and a weak supersolution on every compact time interval before $T_{\max}$.  The assumptions on $q_0$ and $g$ give
\begin{equation*}
0\leq v_\lambda^g(0,\cdot)\leq\overline v_\lambda(0,\cdot), \quad
0\leq v_\lambda^g(t,1)=g(t)\leq g_\lambda(t)=\overline v_\lambda(t,1), \quad \forall t\in[0,T_1].
\end{equation*}
Applying Proposition~\ref{prop:weak-comp} first to $0$ and $v_\lambda^g$, and then to $v_\lambda^g$ and $\overline v_\lambda$, we obtain, for every $0<T_0<\min\{T_1,T_{\max}\}$,
\begin{gather}
0\leq v_\lambda^g(t,r)\leq\overline v_\lambda(t,r)  \quad \forall  (t,r)\in[0,T_0]\times[0,1], \label{eq:ba:compr} \\
0\leq q_\lambda^g(t,r)\leq\overline q_\lambda(t,r) \quad \forall (t,r)\in[0,T_0]\times[0,1]. \label{eq:ba:comp}
\end{gather}

This comparison also shows that $T_{\max}>T_1$.  Indeed, on the core $[0, R]$, the explicit formula for $Q_\lambda$ gives
\begin{equation*}
0\leq\frac{q_\lambda^g(t,r)}r \leq\frac{Q_\lambda(r)}r \leq\frac2\lambda.
\end{equation*}
On the annulus $[R, 1]$, $\overline q_{\lambda,\out}/r$ is bounded on $[0,T_1]\times[R,1]$.  Hence, if $T_{\max}\leq T_1$, estimate \eqref{eq:ba:comp} would imply
\begin{equation*}
\sup_{0\leq t<T_{\max}}\left\|q_\lambda^g(t,\cdot)/r\right\|_{L^\infty(0,1)}
\leq \max\left\{ \frac2\lambda,\, \left\|\overline q_{\lambda,\out}/r\right\|_{L^\infty([0,T_1]\times[R,1])} \right\} <+\infty.
\end{equation*}
This contradicts the continuation alternative in Proposition~\ref{prop:loc-wp}.  Therefore, $T_{\max}>T_1$, and Proposition~\ref{prop:weak-comp} applied on $[0,T_1]$ gives \eqref{eq:ba:compr}--\eqref{eq:ba:comp} on the full cylinder $[0,T_1]\times[0,1]$. It also satisfies 
\begin{equation}\label{eq:cl:actbelow}
0\leq q^g_\lambda(T_1,r)<\pi, \; \;   0\leq v_\lambda^g(T_1,r)\leq\overline v_\lambda(T_1,r)\leq C,  \quad \forall  r\in [0,1], 
\end{equation}

\subsection{Global existence and avoidance of infinite time concentration} \label{subsec:ba:global}
The bound \eqref{eq:cl:actbelow} prevents further concentration through comparison with a stationary harmonic map.  We formulate the resulting continuation argument separately, since it also gives a criterion for excluding infinite-time concentration.

\begin{proposition}[Continuation after terminal clearing]
\label{prop:glo-con}
Let $q=rv$ be a nonnegative maximal mild solution furnished by Proposition~\ref{prop:loc-wp}, with boundary trace $g\in C^1_{\mathrm{loc}}([0,+\infty))$, and fix $T\in(0,T_{\max})$.  Suppose that $q(T,r)<\pi$ for every  $r\in[0,1]$, and that the future trace satisfies  $0\leq g(t)\leq g(T)$ for every  $t\in[T,+\infty)$.
Then the solution is global, namely $T_{\max}=+\infty$. 

If, in addition, $\sup_{t\geq T}|g'(t)|<+\infty$,
then this global solution satisfies
\begin{equation}\label{es:prcon}
\sup_{t\geq T}\|q_r(t,\cdot)\|_{L^\infty(0,1)}<+\infty.
\end{equation}
\end{proposition}

\begin{proof}
Set $v_T(r):= q(T,r)/r$,
where the quotient at $r=0$ is understood by continuous extension, and choose some  $M>\|v_T\|_{L^\infty(0,1)}$.
Recall that $Q_\rho(r)=2\arctan(r/\rho)$.  The function $Q_\rho(r)/r$ is decreasing in $r$, with its limit at $0$ being $2/\rho$.
Choose $r_0\in(0,1)$ such that $\pi/(2r_0)>M$, and then choose $\rho>0$ sufficiently small that $Q_\rho(r_0)>\pi/2$.  It follows that
\begin{equation*}
\frac{Q_\rho(r)}r \geq\frac{Q_\rho(r_0)}{r_0}
> M> \frac{q(T,r)}r, \quad \forall r\in(0,r_0].
\end{equation*}
By continuity, the same inequality holds at $r=0$.
\smallskip

On $[r_0,1]$, the strict inequality $q(T,\cdot)<\pi$ has a uniform gap.  Since $Q_\rho\to\pi$ uniformly on this interval as $\rho\to0^+$, we may reduce $\rho$ so that
$q(T,r)<Q_\rho(r)$ for every  $r\in[r_0,1]$.
Thus, $q(T,\cdot)<Q_\rho$ on $[0,1]$.  In particular, the future trace satisfies
\begin{equation*}
0\leq g(t)\leq g(T)=q(T,1)<Q_\rho(1),
\quad \forall t\in[T,+\infty).
\end{equation*}

By Definition~\ref{def:weak-subsuper}, the profile $v=q/r$ on $[T, T_{\max})$ is both a weak subsolution and a weak supersolution, while $Q_\rho/r$ is a smooth stationary solution.  Proposition~\ref{prop:weak-comp} therefore gives
\begin{equation}\label{eq:cont:vbound}
0\leq v(t,r)\leq \frac{Q_\rho(r)}{r}\leq\frac2\rho,
\quad \forall (t,r)\in[T,T_{\max})\times[0,1],
\end{equation}
moreover, 
\begin{equation}\label{eq:cont:qbound}
0\leq q(t,r)\leq Q_\rho(r),
\quad \forall (t,r)\in[T,T_{\max})\times[0,1],
\end{equation}
If $T_{\max}<+\infty$, then \eqref{eq:cont:vbound}, together with the boundedness of $v$ on $[0,T]$, contradicts the continuation alternative in Proposition~\ref{prop:loc-wp}.  Hence $T_{\max}=+\infty$.
\vspace{2mm}

Assume now that $\sup_{t\geq T}|g'(t)|<+\infty$. We consider the four-dimensional lift as in Appendix \ref{sec:wp-comp}. Define
\begin{equation*}
V(t,x):=v(t,|x|) \textrm{ \; and \; }
U(t,x):=V(t,x)-g(t).
\end{equation*}
Then $U$ has zero Dirichlet boundary value and satisfies
\begin{equation*}
U_t=\Delta_{\mathbb R^4}U+H(t,x),
\qquad
H(t,x):=F\bigl(|x|,U(t,x)+g(t)\bigr)-g'(t).
\end{equation*}
Estimate \eqref{eq:cont:vbound}, the bound on $g'$, and \eqref{eq:rv:react} show that
\begin{equation*}
\sup_{t\geq T} \left(\|U(t)\|_{L^\infty(B^4)} +\|H(t)\|_{L^\infty(B^4)}\right)
<+\infty.
\end{equation*}
For $t\geq T+1$, restart the mild formula at time $t-1$ and consider it on $[t-1, t+1]$. The Dirichlet heat-semigroup gradient estimate used in the proof of \eqref{eq:wp:pos-gra} gives
\begin{align*}
\|\nabla U(t)\|_{L^\infty(B^4)} \leq C\|U(t-1)\|_{L^\infty(B^4)} +C\int_{t-1}^t(t-s)^{-1/2}\|H(s)\|_{L^\infty(B^4)}\dd s
\leq C,
\end{align*}
where $C$ is independent of $t$.  Since $q_r=v+rv_r$ and $|\nabla V|=|\nabla U|$, we obtain
\begin{equation*}
\sup_{t\geq T+1}\|q_r(t,\cdot)\|_{L^\infty(0,1)}<+\infty.
\end{equation*}
Finally, the positive-time estimate \eqref{eq:wp:pos-gra}, applied on $[T, T+1]$, controls the remaining compact time interval.  Combining the two bounds finishes the proof.
\end{proof}

\subsection{Proof of Theorem~\ref{thm:main}}
\label{subsec:ba:main}

Fix the parameters furnished by Lemma~\ref{lem:annular-mode}.  We choose $\lambda_*>0$ sufficiently small.
Let $\lambda\in(0,\lambda_*)$, let $q_0$ satisfy \eqref{eq:ma:data}, and set $b_0:=q_0(1)$.  Then $b_0>\pi$.

\begin{itemize}[leftmargin=2.0em]
\item[(i)]  We first consider the constant boundary trace $b_0$.  Since the boundary trace is constant, Proposition~\ref{prop:wp-reg}\textup{(b)}, applied for any $k\geq1$, shows that the solution is smooth for every $t\in (0, T_{\max})$.  Fix $t_0>0$ small.  Restarting the solution at $t_0$ and applying the CDY criterion Theorem~\ref{thm:in:threshold}\textup{(i)}, we obtain finite-time blow-up.
\vspace{2mm}

\item[(ii)] We next consider the controlled branch.  Fix $g\in\mathcal G_\lambda$ satisfying $g(0)=b_0$.  The estimate \eqref{in:ovqla:T1} in Section~\ref{subsec:ba:rem}, together with the weak supersolution gluing and comparison argument in Section~\ref{subsec:weak-glu}, gives $T_{\max}>T_1$ and \eqref{eq:ba:comp}--\eqref{eq:cl:actbelow} on $[0,T_1]\times[0,1]$. 
    
Since $0\leq g(t)\leq g(T_1)$ for every $t\in[T_1,+\infty)$, Proposition~\ref{prop:glo-con}, applied with $T=T_1$, yields $T_{\max}=+\infty$.  Suppose in addition that \eqref{eq:ma:unider} holds, then  Proposition~\ref{prop:glo-con} gives uniform radial-gradient bound \eqref{es:prcon}, thus excludes infinite-time concentration. For every $\tau\in(0,T_1)$, the positive-time estimate \eqref{eq:wp:pos-gra} gives the corresponding bound on $[\tau,T_1]$.  Combining the two estimates proves \eqref{eq:ma:unigra}. This finishes the proof.
\end{itemize}

\section{Boundary controls producing arbitrarily fast blow-up}  \label{sec:ind-blow}

In this section we prove Theorem~\ref{thm:ind-blow}.
 We divide the proof into three steps: first choose the CDY initial scale $\lambda_0$, then prepare the required ordering by the boundary control, and finally freeze the boundary value and use weak comparison with CDY profile to force the blow-up.

\begin{proof}[Proof of Theorem~\ref{thm:ind-blow}]
Fix a nonnegative function $q_0\in X_{\mathrm{cor}}$ and a time $\mathcal T>0$.  In the proof we use the CDY profile given in Proposition~\ref{lem:cdy-pro}; we fix the parameters therein as $\varepsilon=1/2$, $\mu= 2\mu_0$, and $\delta= \kappa/\mu$.
\smallskip

\noindent\textit{Step 1. Choice of the collapse time and scale.} We first fix the preparation time for ordering as  $t_1= \mathcal T/4$, and the anticipated collapse time of the CDY core after the preparation as $\mathcal T/2$. 
We choose the ``initial'' scale $\lambda_0$ such that collapse time $2\sqrt{\lambda_0}/\delta\leq \mathcal T/2$.

\smallskip

\noindent\textit{Step 2. Preparation of the CDY ordering.} Choose a smooth nondecreasing function $\eta:[0,\infty)\to[0,1]$ such that
\begin{equation*}
\eta(0)=0, \quad \eta(t)>0\quad \forall t\in(0,+\infty), \textrm{ and \; } \eta(t)=1\quad \forall t\in[t_1,+\infty),
\end{equation*}
and such that $\eta$ is flat at $t=0$, namely $\eta^{(m)}(0)= 0$ for every integer $m\geq 1$.  Let $W$ be the radial mild solution in $B^4$ of
\begin{equation*}
\begin{cases}
W_t=\Delta_{\mathbb R^4}W,&\forall (t,x)\in\mathbb R^+\times B^4,\\
W(0,x)=0,&\forall x\in B^4,\\
W(t,x)=\eta(t),&\forall (t,x)\in\mathbb R^+\times\partial B^4.
\end{cases}
\end{equation*}
The parabolic strong maximum principle gives $W(t_1,x)>0$ for every $x\in B^4$, while $W(t_1,x)=1$ for $x\in\partial B^4$.  Since $W(t_1,\cdot)$ is continuous on $\overline{B^4}$, it follows that
\begin{equation*}
c_{t_1}:=\min_{x\in\overline{B^4}}W(t_1,x)>0.
\end{equation*}

We define the control $g(t):=q_0(1)+A\eta(t)$, where $A>0$ is chosen so large that
\begin{equation}\label{eq:ind:amp}
Ac_{t_1}>\frac2{\lambda_0}+\frac2\mu, \qquad q_0(1)+A>\pi+Q_\mu(1).
\end{equation}
The control is smooth and nonnegative, satisfies $g(0)=q_0(1)$, and is constant for $t\geq t_1$.
\smallskip

Let $q=rv$ be the corresponding maximal mild solution from Proposition~\ref{prop:loc-wp}.  Since $g$ is smooth, Proposition~\ref{prop:wp-reg}\textup{(b)} shows that the solution is classical for every positive time before $T_{\max}$.  If $T_{\max}\leq t_1$, the conclusion already follows.  We may therefore assume that $T_{\max}>t_1$.  

Comparison with zero gives $V\geq0$, where $V(t,x):=v(t,|x|)$.  Since
\begin{equation}\label{eq:ind:reasig}
F(r,z)\geq0, \quad \forall (r,z)\in[0,1]\times[0,+\infty),
\end{equation}
the difference $D:=V-AW$ satisfies
\begin{equation*}
D_t-\Delta_{\mathbb R^4}D=F(|x|,V)\geq0.
\end{equation*}
Moreover, $D(0,\cdot)=V_0\geq 0$, and  $D(t,\cdot)|_{\partial B^4}=q_0(1)\geq 0$.
The linear parabolic comparison principle therefore gives $V\geq AW$.  In particular,
\begin{equation}\label{eq:ind:linlow}
v(t_1,r)\geq Ac_{t_1},
\qquad \forall r\in[0,1].
\end{equation}

The elementary bound $\arctan s\leq s$ gives
\begin{equation*}
\frac{Q_{\lambda_0}(r)}r\leq\frac2{\lambda_0} \textrm{ \; and \;} \frac{Q_\mu(r^{3/2})}r \leq\frac{2r^{1/2}}\mu\leq\frac2\mu.
\end{equation*}
Combining these bounds with \eqref{eq:ind:amp}--\eqref{eq:ind:linlow}, we obtain
\begin{equation}\label{eq:ind:iniord}
Q_{\lambda_0}(r)+Q_\mu(r^{3/2})<q(t_1,r), \quad \forall r\in(0,1].
\end{equation}
\smallskip

\noindent\textit{Step 3. Collapse of the prepared CDY core.} For $t\geq t_1$, let $\ell$ solve
\begin{equation}\label{eq:ind:scale}
\ell'=-\delta\ell^{1/2}, \qquad \ell(t_1)=\lambda_0.
\end{equation}
It reaches zero at $t_b:=t_1+2\sqrt{\lambda_0}/\delta<\mathcal T$. By the standard CDY theory (see Proposition~\ref{lem:cdy-pro}),
\begin{equation}\label{eq:ind:dynlow}
\underline q(t,r):=Q_{\ell(t)}(r)+Q_\mu(r^{3/2})
\end{equation}
is a subsolution (and $\underline v$ is a weak subsolution).  Moreover, since $g$ is constant for $t\geq t_1$,
\begin{equation*}
\underline q(t,1)<\pi+Q_\mu(1)<g(t), \quad \forall t\in[t_1,t_b).
\end{equation*}
Combining this boundary ordering with \eqref{eq:ind:iniord} at time $t_1$,  the comparison principle Proposition~\ref{prop:weak-comp} therefore gives $q(t,r)\geq \underline q(t,r)$ on every subinterval of the common existence interval.  

Suppose that $T_{\max}>t_b$.  Since both $q$ and $\underline q$ are continuous, the preceding comparison holds for every $t\in[t_1,t_b)$, and \eqref{eq:cdy:gralow} gives
\begin{equation*}
\|q_r(t,\cdot)\|_{L^\infty(0,1)} \geq\frac{\pi}{2\ell(t)} \longrightarrow+\infty \qquad\text{as }t\to t_b.
\end{equation*}
This contradicts the positive-time gradient estimate
\eqref{eq:wp:pos-gra} on any closed interval containing $t_b$
and contained in $(0,T_{\max})$.
Hence $T_{\max}\leq t_b<\mathcal T$.  By Definition~\ref{def:wp:blowup} and Remark~\ref{rem:wp:cla-blow}, the controlled solution (which is classical) blows up in finite time before $\mathcal T$.  
\end{proof}

\section{Unavoidable blow-up under bounded boundary controls}
\label{sec:un-blow}

This section proves Theorem~\ref{thm:bou-con}.  The proof uses the same collapsing CDY profile as in Section~\ref{sec:ind-blow}, but the ordering is prepared in a different way.  In Section~\ref{sec:ind-blow}, a large boundary control raises the solution above a CDY profile on the whole interval.  Here the initial datum already dominates the CDY profile on a fixed inner region $[0,R]$.  The main issue is therefore to preserve the required ordering at the interface $R$.

Throughout this section, the CDY profile is used only on $[0,R]$. The outer parabolic evolution enters the argument only through the interface value $q(t,R)$.  We prepare the initial datum with a strict positive margin near $R$.  A short-time estimate shows that a boundary control with prescribed amplitude bound cannot remove this interface margin immediately.  We then choose the CDY scale so that the inner core collapses before the interface ordering can be lost.
\vspace{2mm}

\begin{figure}[H]
\centering
\begin{tikzpicture}[
  x=.92cm,
  y=.84cm,
  >=Latex,
  actual/.style={blue!70!black,very thick},
  cdy/.style={red!78!black,very thick},
  guide/.style={black!55,dashed},
  every node/.style={font=\small}
]

\begin{scope}

  \draw[->] (0,0) -- (6.12,0) node[right] {$r$};
  \draw[->] (0,-1.18) -- (0,3.18) node[above] {$q$};
  \draw[guide] (2,-1.10) -- (2,3.02);
  \draw[guide] (0,2.15) -- (2.30,2.15);
  \node[anchor=east] at (-.08,2.15) {$B$};
  \node[below] at (2,0) {$R$};
  \node[below] at (5.8,0) {$1$};
  \node[font=\normalsize] at (2.2,3.67) {$t=0$};
  \node[black!65] at (1.00,-.82) {inner region};
  \node[black!65] at (3.92,-.82) {outer region};

  \draw[cdy]
    plot[smooth] coordinates {
      (0,0)
      (.13,.40)
      (.32,.83)
      (.65,1.29)
      (1.05,1.61)
      (1.50,1.82)
      (2,1.94)
    };

  \draw[actual]
    plot[smooth] coordinates {
      (0,0)
      (.13,.47)
      (.32,.93)
      (.65,1.42)
      (1.05,1.77)
      (1.50,2.10)
      (2,2.55)
      (2.55,2.48)
      (3.25,2.15)
      (4.05,1.50)
      (4.78,.58)
      (5.30,.10)
      (5.8,0)
    };

  \draw[<->,black!70] (2.20,1.96) -- (2.20,2.53);

  \node[anchor=south west] at (2.28,2.61)
    {$q_0(R)>B>\underline q(0,R)$};

  \fill[actual] (2,2.55) circle (1.4pt);
  \fill[cdy] (2,1.94) circle (1.4pt);
\end{scope}

\draw[->,very thick] (6.28,1.55) -- (7.08,1.55);

\node[align=center] at (6.68,2.02)
  {short time\\$\ell(t)\downarrow0$};

\begin{scope}[xshift=7.35cm]

  \draw[->] (0,0) -- (6.12,0) node[right] {$r$};
  \draw[->] (0,-1.18) -- (0,3.48) node[above] {$q$};

  \draw[guide] (2,-1.10) -- (2,3.30);
  \draw[guide] (0,2.15) -- (2.30,2.15);

  \node[anchor=east] at (-.08,2.15) {$B$};
  \node[below] at (2,0) {$R$};
  \node[below] at (5.8,0) {$1$};

  \node[font=\normalsize] at (2.9,3.72)
    {$0<t<\tau_\lambda<\tau_0$};

  \node[black!65] at (1.00,-.82) {inner region};
  \node[black!65] at (3.92,-.82) {outer region};

  \draw[cdy]
    plot[smooth] coordinates {
      (0,0)
      (.04,.60)
      (.10,1.25)
      (.20,1.65)
      (.38,1.84)
      (.75,1.94)
      (1.30,2.01)
      (2,2.06)
    };

  \draw[actual]
    plot[smooth] coordinates {
      (0,0)
      (.04,.72)
      (.10,1.42)
      (.20,1.82)
      (.38,2.02)
      (.75,2.17)
      (1.30,2.34)
      (2,2.53)
      (2.60,2.35)
      (3.30,1.95)
      (4.10,1.35)
      (4.75,.70)
      (5.10,.38)
    };

  \draw[actual,dashed]
    plot[smooth] coordinates {
      (5.10,.38)
      (5.35,.25)
      (5.55,-.05)
      (5.68,-.45)
      (5.76,-.90)
      (5.80,-1.08)
    };

  \draw[guide] (.10,0) -- (.10,1.25);
  \fill[cdy] (.10,1.25) circle (1.4pt);
  \node[anchor=west,align=left] at (.22,.93)
    {$r=\ell(t)$\\$\underline q\geq\pi/2$};

  \draw[<->,black!70] (2.20,2.08) -- (2.20,2.51);
  \node[anchor=south west] at (2.28,2.61)
    {$q(t,R)>B>\underline q(t,R)$};
  \fill[actual] (2,2.53) circle (1.4pt);
  \fill[cdy] (2,2.06) circle (1.4pt);

  \fill[actual] (5.8,-1.08) circle (1.5pt);

  \node[anchor=west,align=left] at (5.93,-.98)
    {$g(t)=-M$\\$M$ large};
\end{scope}

\draw[actual] (3.80,-1.62) -- (4.35,-1.62);
\node[anchor=west] at (4.43,-1.62) {solution};

\draw[cdy] (6.05,-1.62) -- (6.60,-1.62);
\node[anchor=west] at (6.68,-1.62) {CDY lower profile};

\end{tikzpicture}

\caption{Schematic of the localized CDY comparison.  Even under large control $g(t)=-M$, represented by the dashed boundary layer, the interface ordering $q(t, R)>B>\underline q(t, R)$ persists long enough for the inner CDY scale to collapse.}
\label{fig:un:interface}
\end{figure}
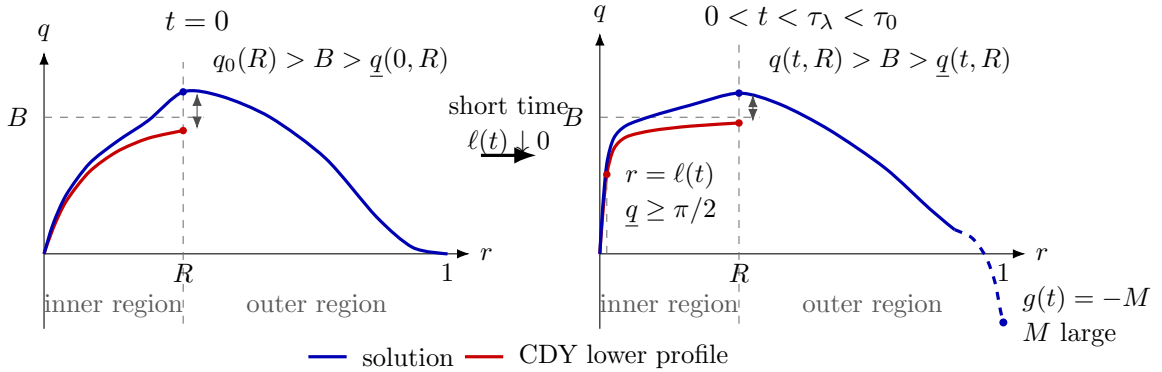

\subsection{A short-time estimate controlling the interface}
We first isolate the following useful lemma to control the interface change.
\begin{lemma} \label{lem:un:interface}
Let $0<R< 1/2$, $B>0$, and $M> 0$.  Suppose that $\rho\in (0, R)$ and  $R K\in (B, +\infty)$.
There exists $\tau_0=\tau_0(R, \rho, B, K, M)>0$ such that the following property holds.

Let $\mathcal{T}> 0$. Let $q_0\in X_{\mathrm{cor}}$ be nonnegative, with $q_0(1)=0$, and write $v_0=q_0/r$.  Suppose that $v_0(r)\geq K$ for every  $r\in[R-\rho,R+\rho]$.
Let $g\in C^1_{\mathrm{loc}}([0,+\infty))$ satisfy $g(0)=0$ and $\|g\|_{L^\infty(0, \mathcal{T})}\leq M$.
Then the corresponding maximal mild solution on $[0, T_{\max})$ satisfies
\begin{equation}\label{eq:un:interface-con}
q(t,R)>B, \quad \forall 0\leq t<\min\{\tau_0, \mathcal{T}, T_{\max}\}.
\end{equation}
In particular, $\tau_0$ is independent of the particular initial profile and boundary control satisfying the preceding assumptions.
\end{lemma}

\begin{proof}
Choose a smooth radial function
$\chi_4\in C_c^\infty(B^4)$ such that
\begin{equation*}
0\leq\chi_4\leq1, \quad \operatorname{supp}\chi_4 \subset\{x\in B^4:R-\rho<|x|<R+\rho\},
\quad \chi_4(Re_1)=1.
\end{equation*}
Let $V_0(x):=v_0(|x|)$. Since $q_0$ is nonnegative and $v_0\geq K$ on
$[R-\rho,R+\rho]$, we have
\begin{equation}\label{eq:un:cutord}
V_0\geq K\chi_4
\qquad\text{on }B^4.
\end{equation}

The purpose of the following function is to provide a lower bound that depends only on  $M$.  For $M>0$, define the spatially constant radial function
\begin{equation*}
y_M(t):= -\frac{M}{\sqrt{1-\frac43M^2t}}, \quad \forall 0\leq t<\frac{3}{4M^2}.
\end{equation*}
A direct differentiation gives
\begin{equation*}
y_M'(t) = -\frac{2M^3}{3} \left(1-\frac43M^2t\right)^{-3/2}
= \frac23y_M(t)^3.
\end{equation*}

Recall that $F(r,0)=0$.  By \eqref{eq:rv:react}, for every $\xi\leq0$,
\begin{align*}
F(r,\xi)= -\int_\xi^0\partial_sF(r,s)\dd s \geq-\int_\xi^02s^2\dd s =\frac23\xi^3.
\end{align*}
Since $y_M(t)\leq0$, and $y_M$ is spatially independent, it follows  that
\begin{equation*}
y_M' -y_{M,rr} -\frac3r y_{M,r} -F(r,y_M)= 
y_M'-F(r,y_M)
\leq0.
\end{equation*}
Thus, $y_M$ is a pointwise subsolution of \eqref{eq:rv:radeq} (and a weak subsolution in the sense of Definition~\ref{def:weak-subsuper}).  Moreover,
$y_M(0)=-M$ and $y_M(t)\leq -M$ is helpful for comparison.

\vspace{2mm}

Comparing $v$ and $y_M$, their initial and boundary orderings satisfy
\begin{equation*}
v_0\geq0\geq y_M(0) \textrm{ \; and \;\; }  g(t)\geq-M\geq y_M(t)\;\; \forall  0<t<\min\{\mathcal T,\frac{3}{4M^2}\}.
\end{equation*}
Proposition~\ref{prop:weak-comp} therefore gives $v(t,r)\geq y_M(t)$ throughout the common existence interval. Next, we further improve this estimate.

Let $V(t,x):=v(t,|x|)$ and $Z(t,x):=V(t,x)-y_M(t)$.
Subtracting the weak subsolution inequality for $y_M$ from the weak equation for $V$, and using the monotonicity of $F$, we obtain
\begin{equation*}
Z_t-\Delta_{\mathbb R^4}Z \geq F(|x|,V)-F(|x|,y_M) \geq 0
\end{equation*}
in the weak sense.  Moreover, on the boundary,
\begin{equation*}
Z(t,x)=g(t)-y_M(t)\geq0, \quad \forall x\in\partial B^4,
\end{equation*}
and \eqref{eq:un:cutord} gives the initial ordering,
\begin{equation*}
Z(0)=V_0+M \geq K\chi_4+M \geq (K+M)\chi_4.
\end{equation*}

Set $H(t,x):=(K+M)(S_{D,4}(t)\chi_4)(x)$, where $S_{D, 4}$ is the Dirichlet heat semigroup on $B^4$. Then $H$ solves the linear 4D heat equation with Dirichlet boundary and initial value $(K+M)\chi_4$.  Since $Z$ is a weak
supersolution of the same equation, 
the linear comparison principle gives
\begin{equation*}
 Z(t, x)\geq (K+M) (S_{D,4}(t)\chi_4)(x), \;\;  \forall  0<t<\min\{\mathcal T, T_{\max}, \frac{3}{4M^2}\}.
\end{equation*}
Thus
\begin{equation}\label{eq:un:semlow}
V(t, x) \geq y_M(t)+(K+M) (S_{D,4}(t)\chi_4)(x),
\end{equation}
as long as the solution and $y_M$ are defined.

Since $\chi_4\in C_0(\overline{B^4})$, the strong continuity of the Dirichlet semigroup gives
\begin{equation*}
    \bigl(S_{D,4}(t)\chi_4\bigr)(Re_1) \longrightarrow\chi_4(Re_1)=1 \quad\text{as }t\to0^+.
\end{equation*}
Moreover,
\begin{equation*}
  R\bigl(y_M(0)+(K+M)\chi_4(Re_1)\bigr)=RK> B.  
\end{equation*}
We may therefore choose $0<\tau_0<\frac{3}{4M^2}$, depending on $(R, \rho, B, K, M)$,
so small that
\begin{equation}\label{eq:un:shitime}
R\left[ y_M(t)+(K+M)\bigl(S_{D,4}(t)\chi_4\bigr)(Re_1) \right]>B, \quad \forall t\in[0, \min \{\tau_0, \mathcal{T}, T_{\max}\}).
\end{equation}

Combining \eqref{eq:un:semlow} and \eqref{eq:un:shitime} proves \eqref{eq:un:interface-con}.
\end{proof}

\subsection{The proof of Theorem~\ref{thm:bou-con}}
Given $\varepsilon_*>0$, $M> 0$, and $\mathcal T>0$.  As in Section~\ref{sec:ind-blow}, we use Proposition~\ref{lem:cdy-pro} with $\varepsilon=1/ 2$. We also fix the interface as $R:=1/10$. 
\smallskip

\noindent\textit{Step 1. The inner CDY profile and the initial ordering.}
For $\mu\geq\mu_0$ and a positive scale function $\ell(t)$ with $\ell(0)= \lambda$,  to be fixed later on, we consider the standard CDY profile restricted to the inner region:
\begin{equation}\label{eq:un:cdycore}
\underline q(t,r):= Q_{\ell(t)}(r)+Q_\mu(r^{3/2}),  \; \textrm{ for }  0\leq r\leq R.
\end{equation}
Correspondingly, we seek an initial profile which agrees on $[0, R]$ with
\begin{equation}\label{eq:un:innerdata}
q_{0,\lambda}(r)= Q_\lambda(r)+ar, \quad \forall r\in [0, R],
\end{equation}
where $a$ is a constant to be chosen. Following Proposition~\ref{lem:cdy-pro}, we define the collapse-rate coefficient and the uniform CDY bound at the interface $R$ by
\begin{equation}\label{eq:un:struct}
\delta:=\frac{\kappa}{\mu} \textrm{ \; and \; } B:=\pi+Q_\mu(R^{3/2}).
\end{equation}
Indeed, for every positive scale $\ell$, one has the uniform bound for $\underline q(t, R)$:
\begin{equation}\label{eq:un:cdy-interface}
Q_\ell(R)+Q_\mu(R^{3/2})<B.
\end{equation}

Our idea of design is to ensure two properties: the initial datum dominates the CDY profile on $[0, R]$, and it has a strict interface margin at $r=R$ which, by Lemma~\ref{lem:un:interface}, persists for a short time.
\begin{itemize}[leftmargin=2.2em]
\item To guarantee {\it the initial ordering on $[0,R]$}, namely $ \underline q(0, r)\leq q_{0, \lambda}(r)< \pi+ \varepsilon_*\;  \forall r\in [0, R]$, observe that
\begin{equation}\label{eq:un:tail}
Q_\mu(r^{3/2})\leq \frac{2r^{3/2}}{\mu}
\leq\frac{2R^{1/2}}{\mu}r, \quad \forall r\in(0,R].
\end{equation}
Thus, it suffices to require $2R^{1/2}/\mu< a< \varepsilon_*$. We close these parameter choices by taking
\begin{equation}\label{eq:un:muchoice}
\mu:=\max\left\{2\mu_0,\frac{8R^{1/2}}{\varepsilon_*}\right\}
\textrm{ \; and \; }  a:=\frac{\varepsilon_*}{2}.
\end{equation}

\item We next fix the {\it interface height at R}.  Since
\begin{equation*}
\frac BR = \frac\pi R+\frac{Q_\mu(R^{3/2})}{R}<  \frac\pi R+\frac{2R^{1/2}}{\mu} < \frac\pi R+a,
\end{equation*}
we may define
\begin{equation}\label{eq:un:Kdef}
K:= \frac12\left( \frac BR+\frac\pi R+a \right),
\end{equation}
which satisfies both  $RK>B$ and  $K<\pi/R+a$. \\
Then, we further fix $0<\rho<R$ and $\lambda_1>0$ small such that for the initial data $q_{0, \lambda}$ there is
\begin{equation}\label{eq:un:intheight}
\frac{Q_\lambda(r)}r+a\geq K, \quad \forall r\in[R-\rho,R+\rho],  \; \forall 0<\lambda\leq\lambda_1.
\end{equation}
\end{itemize}

Finally, we are able to define the initial data $q_{0, \lambda}$ on the whole interval $[0, 1]$. 
Choose a smooth radial cutoff $\zeta$ on $[0, 1]$ such that
\begin{equation*}
0\leq\zeta\leq1, \qquad \zeta=1\; \text{on }[0,R+\rho], \quad \textrm{and } \zeta=0\; \text{ near } r=1,
\end{equation*}
and define
\begin{equation}\label{eq:un:initial}
v_{0,\lambda}(r) :=  \zeta(r)\left(  \frac{Q_\lambda(r)}r+a \right), \quad q_{0,\lambda}(r):=rv_{0,\lambda}(r), \; \forall r\in [0, 1].
\end{equation}
Here $q_{0,\lambda}$ is a smooth nonnegative one-corotational profile and vanishes near $r=1$.
\smallskip

In conclusion, under the above selections of parameters $(\mu, a, \delta, B, K, \rho, \lambda_1)$, for any $\lambda\in (0, \lambda_1)$ we have 
\begin{equation}\label{eq:un:iniord}
    \begin{cases}
    q_{0,\lambda}(r)= Q_\lambda(r)+ar
> Q_\lambda(r)+Q_\mu(r^{3/2}), &\quad \forall r\in (0, R], \\
v_{0,\lambda}(r)\geq K, \quad & \forall r\in[R-\rho,R+\rho],\\
0\leq q_{0,\lambda}(r) \leq\pi+a
<\pi+\varepsilon_*, & \forall r\in [0, 1].
    \end{cases}
\end{equation}

\smallskip

\noindent\textit{Step 2. Short-time preservation of the interface ordering.}
Apply Lemma~\ref{lem:un:interface} with the constants $R,\rho,B,K$, and $M$ fixed above.  We obtain a time
$\tau_0>0$, which is independent of $\lambda\in(0,\lambda_1]$ and of the particular boundary control.

More precisely, for every $g\in C^1_{\mathrm{loc}}([0,+\infty))$ satisfying
\begin{equation*}
g(0)=0, \qquad \|g\|_{L^\infty(0,\mathcal T)}\leq M,
\end{equation*}
the maximal mild solution with initial value $q_{0,\lambda}$ (defined in \eqref{eq:un:initial}) satisfies
\begin{equation}\label{eq:un:intshi}
q(t,R)>B, \quad \textrm{ for }  0\leq t<\min\{\tau_0,T_{\max}, \mathcal T\}.
\end{equation}

\smallskip

\noindent\textit{Step 3. Collapse in the inner CDY region.}
Select $0<\lambda< \min \{\lambda_1, R\}$ so that
\begin{equation}\label{eq:un:colwin}
\tau_\lambda:= \frac{2\sqrt{\lambda}}{\delta} < \min\{\tau_0,\mathcal T\}.
\end{equation}
Fix the initial datum $q_0:=q_{0,\lambda}$ by \eqref{eq:un:initial},
and let $g$ be any boundary control satisfying the assumptions of the theorem.  Denote the corresponding maximal mild solution by $q$.

If $T_{\max}\leq\tau_\lambda$, then $T_{\max}<\mathcal T$, and the conclusion follows.  We may therefore suppose that
$T_{\max}>\tau_\lambda$.
By solving the scale equation
\begin{equation*}
\ell'=-\delta\ell^{1/2} \textrm{ \; with \; } \ell(0)=\lambda,
\end{equation*}
we obtain
\begin{equation}\label{eq:un:scale}
\ell(t):= \left(\sqrt{\lambda}-\frac{\delta t}{2}
\right)^2, \; \textrm{ for }  0\leq t<\tau_\lambda.
\end{equation}
On the inner region $[0,R]$, set
\begin{equation}\label{eq:un:insub}
\underline q(t,r) := Q_{\ell(t)}(r)+Q_\mu(r^{3/2}).
\end{equation}
Proposition~\ref{lem:cdy-pro} shows that $\underline q/r$ is a weak subsolution on every compact time interval contained in $[0,\tau_\lambda)$.

\smallskip
\smallskip

We would like to compare $q/r$ and $\underline q/r$ in $[0, R]$\footnote{Indeed, the proof of Proposition~\ref{prop:weak-comp} applies on $(0, R)$, with the weighted spaces defined on $(0,R)$ and with zero test trace at $r=R$. }.   Clearly, $\underline q/r$ is weak subsolution and $q/r$ is a weak solution in the sense of Definition~\ref{def:weak-subsuper} with $[0, 1]$ replaced by $[0, R]$.
At the initial time, \eqref{eq:un:iniord} gives
\begin{equation*}
\underline q(0,r)< q_{0, \lambda}(r)= q_0(r), \quad \forall r\in (0, R].
\end{equation*}
At the interface, \eqref{eq:un:cdy-interface}, \eqref{eq:un:intshi}, and \eqref{eq:un:colwin} give
\begin{equation*}
\underline q(t,R)<B<q(t,R), \quad \forall t\in [0, \tau_\lambda).
\end{equation*}
Thus Proposition~\ref{prop:weak-comp} gives that for every $T\in (0, \tau_{\lambda})$,
\begin{equation}\label{eq:un:innercomp0}
\underline q(t,r)/r\leq q(t,r)/r  \quad\text{for a.e. }(t,r)\in(0, T)\times(0,R).
\end{equation}
Since both profiles are continuous, the comparison holds pointwise, and therefore
\begin{equation}\label{eq:un:innercomp}
\underline q(t,r)\leq q(t,r)  \quad\text{for every }(t,r)\in(0, \tau_{\lambda})\times(0,R).
\end{equation}

Since $\ell(t)\leq\lambda<R$, we may evaluate \eqref{eq:un:innercomp} at $r=\ell(t)$.  This gives
\begin{equation*}
q(t,\ell(t)) \geq Q_{\ell(t)}(\ell(t)) +Q_\mu(\ell(t)^{3/2}) \geq\frac\pi2.
\end{equation*}
Since $q(t,0)=0$, the mean value theorem yields
\begin{equation*}
\|q_r(t,\cdot)\|_{L^\infty(0,R)} \geq \frac{\pi}{2\ell(t)} \longrightarrow+\infty \qquad\text{as }t\to\tau_\lambda^-.
\end{equation*}
On the other hand, the assumption $T_{\max}>\tau_\lambda$ allows us to choose $0<\tau<\tau_\lambda<S<T_{\max}$.
The positive-time gradient estimate \eqref{eq:wp:pos-gra} on $[\tau, S]$ gives a contradiction. Hence $T_{\max}\leq\tau_\lambda<\mathcal T$. Therefore, the maximal mild solution blows up before $\mathcal T$.  Finally, \eqref{eq:un:iniord} gives
\begin{equation*}
\|q_0\|_{L^\infty(0,1)} \leq\pi+\varepsilon_*,
\end{equation*}
which finishes the proof.

\appendix
\section{Well-posedness, blow-up, and weak comparison}   \label{sec:wp-comp}

In this appendix, we establish the analytic framework used throughout the paper. In Sections \ref{subsec:reg-var}--\ref{subsec:wp-mild} we first present the mild-solution theory, continuation criterion, and blow-up terminology through the four-dimensional radial lift formulation.  We then, in Section \ref{subsec:weak-comp}, give the weak comparison principle and the fixed-interface gluing lemma for weak (super, sub)-solutions, whose favorable interface flux is central to our construction.

\subsection{The auxiliary four-dimensional lift} \label{subsec:reg-var}

Let $q(t,r)$ denote the one-corotational angle, with $q(t,0)=0$, and define the radial profile of the lifted variable by $v(t,r):=q(t,r)/r$ for $r\in(0,1]$.  Whenever the limit exists, we set $v(t,0):=\lim_{r\to0^+}q(t,r)/r$.  Although $v$ depends only on the scalar radius $r$, it determines the radial function $V(t,x):=v(t,|x|)$ for $x\in B^4$. Thus, $v(t,\cdot)\in C([0,1])$ is equivalent to $V(t,\cdot)\in C_{\mathrm{rad}}(\overline{B^4})$.  For a classical one-corotational angle, $q(t,0)=0$ also gives $v(t,0)=q_r(t,0)$.   Throughout this appendix, $q$ denotes the one-corotational angle, $v$ its lifted radial profile, and $V$ the corresponding radial function on $B^4$.

 We use the above dimension-raising method to formulate the boundary-value mild theory; this technique is standard \cite{Angenent-Hulshof-Matano-2009,Jendrej-Lawrie-2023}. The reason for introducing this lift is analytic.  Equation \eqref{eq:flow:radial} can be written as
\begin{equation*}
q_t-q_{rr}-\frac1r q_r+\frac q{r^2} =\frac{2q-\sin(2q)}{2r^2}.
\end{equation*}
A direct weak theory in the $q$-variable must therefore handle an inverse-square potential together with the weighted trace and flux at the origin.  Nevertheless, observe that,
\begin{equation}\label{eq:rv:trans}
q_t-\cL q =r\left(v_t-v_{rr}-\frac3r v_r-F(r,v)\right),
\end{equation}
where, for $r\in(0,1]$ there is
\begin{equation}\label{eq:rv:react}
\begin{aligned}
F(r,z)&:=\frac{2rz-\sin(2rz)}{2r^3},\\
F(r,z)&=\frac23z^3-\frac{2}{15}r^2z^5+O(r^4z^7) \quad\text{as }r\to0^+,\\
0\leq\partial_zF(r,z) &=\frac{1-\cos(2rz)}{r^2}\leq2z^2.
\end{aligned}
\end{equation}
The power series of the sine shows that the second line holds locally uniformly for bounded $z$ and that $(x,z)\mapsto F(|x|,z)$ extends smoothly across $x=0$.  The last line gives a uniform Lipschitz constant on every bounded range of $z$. Thus, the one-dimensional radial profile $v$ satisfies
\begin{equation}\label{eq:rv:radeq}
v_t=v_{rr}+\frac3r v_r+F(r,v), \qquad \forall r\in(0,1).
\end{equation}
 Consequently, the four-dimensional function $V$ satisfies
\begin{equation}\label{eq:rv:eq}
V_t=\Delta_{\mathbb R^4}V+F(|x|,V).
\end{equation}

Let $C_{\mathrm{rad}}(\overline{B^4})$ denote the space of continuous radial functions on $\overline{B^4}$, and recall the one-corotational data space $X_{\mathrm{cor}}$ introduced in the Introduction.  The correspondence
\begin{equation*}
q(r)=r v(r) \quad\longleftrightarrow\quad V(x)=v(|x|)
\end{equation*}
is a bijection between $X_{\mathrm{cor}}$ and $C_{\mathrm{rad}}(\overline{B^4})$, and it preserves the outer boundary data. For profiles having the corresponding regularity, the calculation above gives an exact equivalence between the $q$-equation in the one-corotational regularity class and the four-dimensional equation \eqref{eq:rv:eq}.  In particular, regularity across $r=0$ is encoded by the regularity of the radial function $V$ at $x=0$.  

\subsection{Well-posedness and blow-up for mild solutions} \label{subsec:wp-mild}
We first define the mild solution class for functions $V$ (and therefore for $q$) used below.
\begin{definition}\label{def:V:milso}
Given radial initial data $V_0\in C_{\mathrm{rad}}(\overline{B^4})$ and a boundary value $g\in C^1([0,T])$ such that $V_0=g(0)$ on $\partial B^4$, set $U_0:=V_0-g(0)\in C_0(\overline{B^4})$.  For $U:=V-g$, equation \eqref{eq:rv:eq} becomes
\begin{equation}\label{eq:wp:lifeq}
\left\{
\begin{aligned}
U_t&=\Delta_{\mathbb R^4}U+F\bigl(|x|,U+g(t)\bigr)-g'(t), &&\forall x\in B^4,\\
U&=0,&&\forall x\in\partial B^4,\\
U(0)&=U_0.
\end{aligned}
\right.
\end{equation}

We say that $V=U+g$ is a radial mild solution of \eqref{eq:rv:eq} if $U\in C([0,T];C_0(\overline{B^4}))$ and
\begin{equation}\label{eq:wp:duh}
U(t)=S_{D,4}(t)U_0 +\int_0^tS_{D,4}(t-s) \left[F\bigl(|x|,U(s)+g(s)\bigr)-g'(s)\right]\dd s,
\end{equation}
where $S_{D,4}(t)$ is the Dirichlet heat semigroup on $B^4$. 

The radial profile $v(t,r)$ of $V$ then determines $q(t,r):=rv(t,r)$; we call $q$ the mild solution of the original one-corotational equation \eqref{eq:flow:radial} induced by $V$.
\end{definition}

We now record local well-posedness,  the continuation alternative, and higher regularity of the maximal mild solution.  Its proof is rather standard; we include it for readers' convenience. 
For the classical regularity theory for parabolic equations, see \cite{Lieberman-1996}.
\begin{proposition}
\label{prop:loc-wp}
Let $q_0\in X_{\mathrm{cor}}$, set $v_0(r):=q_0(r)/r$, and define $V_0(x):=v_0(|x|)$ on $\overline{B^4}$.  Let $g\in C^1_{\mathrm{loc}}([0,\infty))$ satisfy $g(0)=q_0(1)$.  Then the following statements hold.
\begin{enumerate}[label=\textup{(\roman*)},leftmargin=2.2em]
\item There exists a unique maximal radial mild solution $V\in C\bigl([0,T_{\max});C_{\mathrm{rad}}(\overline{B^4})\bigr)$ of \eqref{eq:rv:eq}, with boundary value $V(t,x)=g(t)$ for $x\in\partial B^4$ and initial value $V_0$.  Its radial profile $v(t,r)$ determines the unique induced mild solution $q(t,r)$ of the one-corotational equation.

\item  Concerning the maximal time, if $T_{\max}<\infty$, then
\begin{equation}\label{eq:wp:blowalt}
\limsup_{t\to T_{\max}} \|V(t)\|_{L^\infty(B^4)} =\limsup_{t\to T_{\max}} \left\|q(t,\cdot)/r\right\|_{L^\infty(0,1)} =+\infty.
\end{equation}
Equivalently, if a mild solution $V$ is defined on $[0,T_0)$ for some finite $T_0$ and satisfies
\begin{equation*}
\sup_{0\leq t<T_0}\|V(t)\|_{L^\infty(B^4)}<\infty,
\end{equation*}
then the unique mild solution extends beyond time $T_0$.

\item The mild solution $V$ has a bounded spatial gradient on every positive-time closed cylinder and is smooth in the interior.  More precisely, for every $0<\tau<S<T_{\max}$,
\begin{equation}\label{eq:wp:pos-gra}
\sup_{\tau\leq t\leq S} \left( \|\nabla V(t)\|_{L^\infty(B^4)} +\|q_r(t,\cdot)\|_{L^\infty(0,1)} \right)<+\infty.
\end{equation}
Moreover, for every compact set $K\Subset B^4$ and every $0<\tau<S<T_{\max}$, one has $V\in C^\infty\bigl([\tau,S]\times K\bigr)$.
Consequently, for every $\rho\in(0,1)$, one has $q\in C^\infty\bigl([\tau,S]\times[0,\rho]\bigr)$.
\end{enumerate}
\end{proposition}

With higher regularity on initial data or control, the regularity up to the initial time and the boundary improves as follows.
\begin{proposition}[Regularity improvements]\label{prop:wp-reg}
Let $q_0$, $g$, $V$, and $q$ be as in Proposition~\ref{prop:loc-wp}.  Then the following statements hold.
\begin{enumerate}[label=\textup{(\alph*)},leftmargin=2.2em]
\item If, in addition,  $V_0\in C^1(\overline{B^4})$, then, for every $S<T_{\max}$,
\begin{equation*}
\sup_{0\leq t\leq S} \left(\|\nabla V(t)\|_{L^\infty(B^4)}+  \|q_r(t,\cdot)\|_{L^\infty(0,1)} \right)<+\infty.
\end{equation*}
Note that only the \emph{zeroth-order corner compatibility} $g(0)=q_0(1)$ is assumed here.
\item Let $\alpha\in(0,1)$ and let $k\geq1$ be an integer.  If $g\in C_{\mathrm{loc}}^{k+\alpha/2}((0,+\infty))$,
then, for every $0<\tau<S<T_{\max}$ one has
\begin{equation*}
V\in C^{k+\alpha/2,\,2k+\alpha}
\bigl([\tau,S]\times\overline{B^4}\bigr) \textrm{ \; and \; }
q\in C^{k+\alpha/2,\,2k+\alpha}
\bigl([\tau,S]\times[0,1]\bigr).
\end{equation*}
In particular, the case $k=1$ shows that the mild solution is classical for every positive time before $T_{\max}$.

\item Under the assumptions of \textup{(b)}, suppose in addition that
\begin{equation*}
V_0\in C^{2k+\alpha}(\overline{B^4}) \textrm{ \; and \; } g\in C_{\mathrm{loc}}^{k+\alpha/2}([0,+\infty)),
\end{equation*}
and that the corresponding compatibility conditions hold at $t=0$.  Then the conclusion in \textup{(b)} also holds with $\tau=0$.
\end{enumerate}
\end{proposition}

\begin{proof}[Proof of Propositions~\ref{prop:loc-wp} and \ref{prop:wp-reg}]
\noindent\textit{Step 1. Local mild well-posedness.} Set $U:=V-g$.  The compatibility $g(0)=q_0(1)$ gives $U_0\in C_0(\overline{B^4})$, while \eqref{eq:rv:react} shows that the source in \eqref{eq:wp:lifeq} is locally Lipschitz from bounded subsets of $C_0(\overline{B^4})$ into $C(\overline{B^4})$, uniformly on compact time intervals.  The semigroup $S_{D,4}(t)$ is an $L^\infty$ contraction and maps $C(\overline{B^4})$ into $C_0(\overline{B^4})$ for $t\in(0,+\infty)$.  Hence the Duhamel map in \eqref{eq:wp:duh} is well defined on $C([0,T];C_0(\overline{B^4}))$.  Its Duhamel term is $O(t)$ at the initial time, while continuity at positive times follows by splitting off a short terminal interval.  The same estimates give a contraction for sufficiently small $T$.  This proves local existence and uniqueness.

\smallskip
\noindent\textit{Step 2. Continuation and $T_{\max}$.} Suppose that $T_{\max}<\infty$ and that $V$ remains bounded.  Fix $t_1\in(0,T_{\max})$.
\begin{equation}\label{eq:wp:mild}
U(t)=S_{D,4}(t-t_1)U(t_1) +\int_{t_1}^tS_{D,4}(t-s) \left[F\bigl(|x|,U(s)+g(s)\bigr)-g'(s)\right]\dd s.
\end{equation}
Write the integrand as $G(s)$.  It is uniformly bounded in $C(\overline{B^4})$.  The part with $s\leq T_{\max}-\eta$ has a positive semigroup lag, whereas the contraction estimate gives
\begin{equation*}
\left\|\int_{T_{\max}-\eta}^tS_{D,4}(t-s)G(s)\dd s\right\|_{L^\infty} \leq \eta\sup_{t_1\leq s<T_{\max}}\|G(s)\|_{L^\infty}.
\end{equation*}
Letting first $t\to T_{\max}$ and then $\eta\to0^+$ shows that $U(t)$ is Cauchy in $C_0(\overline{B^4})$.  Restarting the local mild theory from its limit contradicts maximality.  This proves the continuation alternative.  Since $\|V(t)\|_{L^\infty(B^4)} =\|q(t,\cdot)/r\|_{L^\infty(0,1)}$, it also proves \eqref{eq:wp:blowalt}.
\smallskip

\smallskip
\noindent\textit{Step 3. Positive-time gradient and interior regularity.}
Fix $0<\tau<S<T_{\max}$ and set $t_0:=\tau/2$.  Restarting the mild formula at $t_0$ and using the Dirichlet heat-semigroup estimate
\begin{equation*}
\|\nabla S_{D,4}(t)f\|_{L^\infty(B^4)}
\leq Ct^{-1/2}\|f\|_{L^\infty(B^4)},
\end{equation*}
we obtain
\begin{equation*}
\sup_{\tau\leq t\leq S}\|\nabla V(t)\|_{L^\infty(B^4)}
\leq C_{\tau,S} \left(\|U(t_0)\|_{L^\infty(B^4)} +\|F(|x|,V)-g'\|_{L^\infty((t_0,S)\times B^4)} \right)<+\infty.
\end{equation*}
Since $V$ is bounded on $[t_0,S]$, radiality and $q_r=v+rv_r$ give
\begin{equation*}
\|q_r(t,\cdot)\|_{L^\infty(0,1)}\leq\|V(t)\|_{L^\infty(B^4)}+\|\nabla V(t)\|_{L^\infty(B^4)}.
\end{equation*}
This proves \eqref{eq:wp:pos-gra}.  The interior conclusion follows from the standard parabolic bootstrap, since $(x,z)\mapsto F(|x|,z)$ is smooth.  This proves \textup{(iii)}.

\smallskip
\noindent\textit{Step 4. Spatial gradient up to the initial time.}
Assume that $V_0\in C^1(\overline{B^4})$ and fix $S\in(0,T_{\max})$.  Then $U_0=V_0-g(0)$ belongs to $C^1(\overline{B^4})$ and has zero boundary trace.  The Dirichlet heat semigroup satisfies
\begin{equation*}
\sup_{0\leq t\leq S}
\|\nabla S_{D,4}(t)U_0\|_{L^\infty(B^4)}
\leq C_S\|U_0\|_{C^1(\overline{B^4})}.
\end{equation*}
Set
\begin{equation*}
G(t,x):=F\bigl(|x|,U(t,x)+g(t)\bigr)-g'(t).
\end{equation*}
The continuity of $V$ and the standing assumption $g\in C^1_{\mathrm{loc}}([0,+\infty))$ imply that $G$ is bounded on $[0,S]\times\overline{B^4}$.  Combining the preceding estimate, the semigroup gradient estimate used in Step~3, and the Duhamel formula, we obtain
\begin{equation*}
\|\nabla V(t)\|_{L^\infty(B^4)}
\leq
C_S\|U_0\|_{C^1(\overline{B^4})}
+C\int_0^t(t-s)^{-1/2}\|G(s)\|_{L^\infty(B^4)}\dd s.
\end{equation*}
Since $(t-s)^{-1/2}$ is integrable, this gives
\begin{equation*}
\sup_{0\leq t\leq S}
\|\nabla V(t)\|_{L^\infty(B^4)}<+\infty.
\end{equation*}
Finally, radiality and $q_r=v+rv_r$ yield
\begin{equation*}
\|q_r(t,\cdot)\|_{L^\infty(0,1)}
\leq
\|V(t)\|_{L^\infty(B^4)} +\|\nabla V(t)\|_{L^\infty(B^4)}.
\end{equation*}
This proves \textup{(a)}.

\smallskip
\noindent\textit{Step 5. Positive-time parabolic regularity.}
Fix $\alpha\in(0,1)$, an integer $k\geq1$, and $0<\tau<S<T_{\max}$.  Choose $0<\tau_0<\tau_1<\tau<S<S_1<T_{\max}$.  The solution is bounded on $[\tau_0,S_1]\times\overline{B^4}$.  Since $g,g'$ are also bounded on $[\tau_0,S_1]$, the source
\begin{equation*}
H(t,x):=F\bigl(|x|,U(t,x)+g(t)\bigr)-g'(t)
\end{equation*}
in \eqref{eq:wp:lifeq} belongs to $L^p((\tau_0,S_1)\times B^4)$ for every finite $p$.

Let $\chi\in C^\infty([\tau_0,S_1])$ vanish near $\tau_0$ and satisfy $\chi=1$ on $[\tau_1,S_1]$.  Then $\widetilde U:=\chi U$ has zero initial and boundary values and satisfies
\begin{equation*}
\widetilde U_t-\Delta_{\mathbb R^4}\widetilde U =\chi H+\chi' U\in L^p((\tau_0,S_1)\times B^4).
\end{equation*}
The interior and boundary $W^{2,1}_p$ estimates, followed by the parabolic Sobolev embedding, therefore give $V\in C^{\alpha/2,\alpha}
\bigl([\tau_1,S]\times\overline{B^4}\bigr)$, after choosing $p$ sufficiently large.  Since $(x,z)\mapsto F(|x|,z)$ is smooth by \eqref{eq:rv:react}, we obtain $F(|x|,V)\in C^{\alpha/2,\alpha}
\bigl([\tau_1,S]\times\overline{B^4}\bigr)$.

Using $g\in C_{\mathrm{loc}}^{k+\alpha/2}((0,+\infty))$, the interior and boundary Schauder estimates for \eqref{eq:wp:lifeq}, followed by the usual semilinear bootstrap, now yield
$V\in C^{k+\alpha/2,\,2k+\alpha}
\bigl([\tau,S]\times\overline{B^4}\bigr)$.
Since $V$ is radial, its radial profile has the corresponding regularity on $[0,1]$.  Hence
$q\in C^{k+\alpha/2,\,2k+\alpha}
\bigl([\tau,S]\times[0,1]\bigr)$.
In particular, the case $k=1$ gives the positive-time classical regularity.  This proves \textup{(b)}.

\smallskip
\noindent\textit{Step 6. Regularity at the initial corner.}
Under the assumptions of \textup{(c)}, the initial datum, boundary value, and nonlinear source have the regularity required by the parabolic boundary Schauder theorem, and the compatibility conditions remove the initial boundary corner singularities.  Applying the same Schauder bootstrap on $[0,S]\times\overline{B^4}$ gives
\begin{equation*}
V\in C^{k+\alpha/2,\,2k+\alpha}
\bigl([0,S]\times\overline{B^4}\bigr)
 \textrm{ \;  and \; }
q\in C^{k+\alpha/2,\,2k+\alpha}
\bigl([0,S]\times[0,1]\bigr),
\end{equation*}
for every $S<T_{\max}$.  This proves \textup{(c)} and finishes the proof.
\end{proof}

\begin{remark}[The finite maximal time $T_{\max}$]\label{rem:wp:finmax}
Under the assumptions of Proposition~\ref{prop:loc-wp}, suppose that $T_{\max}<+\infty$.  Then $T_{\max}$ simultaneously shares the following three properties.
\begin{itemize}[leftmargin=2.4em]
\item The solution satisfies
\begin{equation*}
V\in C\bigl([0,T_{\max});C_{\mathrm{rad}}(\overline{B^4})\bigr)  \textrm{ \; and \; }
q/r\in C\bigl([0,T_{\max});C([0,1])\bigr),
\end{equation*}
but neither trajectory admits a continuous extension to $t=T_{\max}$ in the indicated space.

\item The mild-solution norm is bounded on every compact subinterval of $[0,T_{\max})$, but
\begin{equation*}
\limsup_{t\to T_{\max}}
\|V(t)\|_{L^\infty(B^4)}
=
\limsup_{t\to T_{\max}}
\left\|q(t,\cdot)/r\right\|_{L^\infty(0,1)}
=+\infty.
\end{equation*}

\item The radial gradient is bounded on every compact positive-time interval before $T_{\max}$, but
\begin{equation*}
\limsup_{t\to T_{\max}}
\|q_r(t,\cdot)\|_{L^\infty(0,1)}
=+\infty.
\end{equation*}
\end{itemize}
The continuity in the first item follows from Proposition~\ref{prop:loc-wp}\textup{(i)}, while the second item is Proposition~\ref{prop:loc-wp}\textup{(ii)} and also excludes the continuous extensions in the first item.  Finally, Proposition~\ref{prop:loc-wp}\textup{(iii)} gives the preterminal gradient bound, and $q(t,0)=0$ gives
\begin{equation*}
\left\|q(t,\cdot)/r\right\|_{L^\infty(0,1)}
\leq
\|q_r(t,\cdot)\|_{L^\infty(0,1)}.
\end{equation*}
This proves the third item. Finally, we simply recall for the corresponding one-corotational map $u$, one also has
\begin{equation*}
\|q_r(t,\cdot)\|_{L^\infty(0,1)}
\leq\|\nabla u(t)\|_{L^\infty(\D)}
\leq\sqrt2\,\|q_r(t,\cdot)\|_{L^\infty(0,1)}.
\end{equation*}
\end{remark}

\begin{definition}[Blow-up for mild solutions] \label{def:wp:blowup}
Let $q_0\in X_{\mathrm{cor}}$, and let $g\in C^1_{\mathrm{loc}}([0,\infty))$ satisfy $g(0)=q_0(1)$. 
Let $V$ be the maximal mild solution in Proposition~\ref{prop:loc-wp}, and let $q(t,r)=rv(t,r)$ be the induced one-corotational mild solution.
\begin{enumerate}[label=\textup{(\roman*)},leftmargin=2.4em]
\item We say that $q$ \emph{blows up in finite time} if $T_{\max}<+\infty$.

\item We say that $q$ develops \emph{infinite-time concentration} if $T_{\max}=+\infty$ and
\begin{equation*}
\limsup_{t\to+\infty}
\|q_r(t,\cdot)\|_{L^\infty(0,1)}=+\infty.
\end{equation*}
\end{enumerate}
\end{definition}

\begin{remark}[Relation between mild and classical solutions]\label{rem:wp:cla-blow}
The mild formulation accommodates lower regularity and the initial boundary corner.  We say that a classical solution attains $q_0$ in $X_{\mathrm{cor}}$ if
\begin{equation*}
\left\|q(t,\cdot)/r-q_0/r\right\|_{L^\infty(0,1)} \longrightarrow 0
\text{ \; as }t\to 0^+.
\end{equation*}
Any such classical solution is also a mild solution.  Indeed, its four-dimensional lift belongs locally in positive time to the energy class.  Since $\{0\}$ has zero $H^1$-capacity in $B^4$, the equation extends weakly across the origin. 
The weak--semigroup equivalence gives the restarted Duhamel formula on every $[\tau,S]$.  Letting $\tau\to0^+$ and using the prescribed convergence to $q_0$ in $X_{\mathrm{cor}}$, we obtain \eqref{eq:wp:duh}.  Hence, uniqueness identifies the classical solution with the maximal mild solution on their common lifespan.

Conversely, if $g\in C_{\mathrm{loc}}^{1+\alpha/2}((0,+\infty))$ for some $\alpha\in(0,1)$, Proposition~\ref{prop:wp-reg}\textup{(b)} shows that the maximal mild solution is classical for every positive time.  Thus, the maximal mild and classical lifespans agree.  Remark~\ref{rem:wp:finmax} then shows that finite-time blow-up is characterized by
\begin{equation*}
\limsup_{t\to T_{\max}} \|q_r(t,\cdot)\|_{L^\infty(0,1)}= +\infty.
\end{equation*}
Consequently,  Definition~\ref{def:wp:blowup} concerning finite-time blow-up of the maximal mild solution agrees with the usual notion of gradient blow-up for classical solutions.
\end{remark}

\begin{remark}
This terminology concerns the breakdown of the maximal mild solution and, when applicable, of the corresponding maximal classical solution.  It does not exclude a weak continuation beyond a finite singular time; see Struwe \cite{Struwe-1985} for his global weak theory on surfaces. 

Moreover, the variational class introduced in the next section is used only for comparison and is not such a post-singularity continuation class.
\end{remark}

\subsection{Weak comparison and fixed-interface gluing} \label{subsec:weak-comp}

The preceding subsection develops the mild-solution theory for the controlled flow.  The purpose of this subsection is to compare these solutions with explicit barriers that may be only piecewise classical and may carry derivative jumps at fixed interfaces. 
We now introduce a variational framework in order to compare these solutions.  This framework is used only for comparison and is not intended as a continuation theory beyond finite-time singularities.

For convenience in the applications, we formulate this framework in terms of the radial profile $v(t,r)$.  Under the identification $V(t,x)=v(t,|x|)$, it is equivalent to the standard weak formulation for radial functions on $B^4$.  The natural radial measure inherited from $B^4$ is $r^3\dd r$.  We therefore define
\begin{equation*}
\mathcal H:=L^2\bigl((0,1),r^3\dd r\bigr),\quad \mathcal V:=\{z\in\mathcal H:z_r\in\mathcal H\}, \textrm{ \; and \;} \mathcal V_0:=\{z\in\mathcal V:z(1)=0\}.
\end{equation*}
Although this weight vanishes at the origin, it is only the Jacobian of four-dimensional polar coordinates. The resulting formulation is the radial representation of the uniformly parabolic equation on $B^4$ and does not rely on one-dimensional uniform ellipticity.
Here the ordinary one-dimensional $H^1$ regularity gives the trace in the definition of $\mathcal V_0$.  The same argument on a neighborhood of each fixed $R\in(0,1)$ makes evaluation at $R$ a bounded linear functional on $\mathcal V$.  For $S>0$, set
\begin{equation}\label{eq:wc:energy}
\mathcal W_S :=C([0,S];\mathcal H) \cap L^2(0,S;\mathcal V) \cap L^\infty((0,S)\times(0,1)).
\end{equation}

\begin{lemma}[The weighted energy class] \label{lem:mild-energy}
Let $V$ be the maximal radial mild solution from Proposition~\ref{prop:loc-wp}, and let $v$ be its one-dimensional radial profile.  For every $S<T_{\max}$, one has
\begin{equation}\label{eq:me:class}
v\in\mathcal W_S \quad\text{and}\quad v_t\in L^2(0,S;\mathcal V_0').
\end{equation}
Moreover, for every $\phi\in L^2(0,S;\mathcal V_0)$,  $v$ satisfies the following variational form of \eqref{eq:rv:radeq}:
\begin{equation}\label{eq:me:var}
\int_0^S\langle v_t,\phi\rangle_{\mathcal V_0',\mathcal V_0}\dd t +\int_0^S\int_0^1v_r\phi_r r^3\dd r\dd t -\int_0^S\int_0^1F(r,v)\phi r^3\dd r\dd t=0.
\end{equation}
\end{lemma}

\begin{proof}
Set $U:=V-g$ as in \eqref{eq:wp:lifeq}.  Since $U(0)\in C_0(\overline{B^4})\subset L^2(B^4)$ and $V$ is bounded on $[0,S]$, the source term in the lifted equation belongs to $L^2((0,S)\times B^4)$.  The standard energy estimate for the Dirichlet heat equation gives
\begin{equation*}
U\in C([0,S];L^2(B^4)) \cap L^2(0,S;H_0^1(B^4)), \qquad U_t\in L^2(0,S;H^{-1}(B^4)).
\end{equation*}
The semigroup solution coincides with this energy solution by uniqueness. Adding back the spatially constant lift $g(t)$ gives the corresponding energy regularity for $V$.  Passing to radial coordinates then yields \eqref{eq:me:class} and the asserted variational identity for the profile $v$.
\end{proof}

\begin{definition}[Weak subsolutions and supersolutions]
\label{def:weak-subsuper}
A function $v\in\mathcal W_S$ is called a weak subsolution of \eqref{eq:rv:radeq} if $v_t\in L^2(0,S;\mathcal V_0')$ and
\begin{equation}\label{eq:wc:subdef}
\int_0^S\langle v_t,\phi\rangle_{\mathcal V_0',\mathcal V_0}\dd t +\int_0^S\int_0^1v_r\phi_r r^3\dd r\dd t -\int_0^S\int_0^1F(r,v)\phi r^3\dd r\dd t \leq0
\end{equation}
for every nonnegative $\phi\in L^2(0,S;\mathcal V_0)$.  A weak supersolution is defined by reversing the inequality.  The time derivative and the test functions may equivalently be localized in time; this gives the same definition on every subinterval of $(0,S)$.

In particular, by Lemma~\ref{lem:mild-energy}, the radial profile of every maximal mild solution from Proposition~\ref{prop:loc-wp} may be used both as a weak subsolution and as a weak supersolution.
\end{definition}

The point $r=0$ is the center of $B^4$, not a spatial boundary.  The cutoff argument in the proof of Lemma~\ref{lem:fix-int} shows that radial functions in $\mathcal V_0$ that vanish near $r=0$ are dense in $\mathcal V_0$. This density property reflects the fact that $\{0\}$ has zero $H^1$-capacity in $B^4$. Consequently, neither a boundary trace nor a separate origin-flux condition is imposed at $r=0$.  At $r=1$, weak subsolutions and supersolutions have their usual Sobolev traces, which need not agree.

\begin{proposition}[Weak comparison]\label{prop:weak-comp}
Let $\underline v,\overline v\in\mathcal W_S$ be respectively a weak subsolution and a weak supersolution in the sense of Definition~\ref{def:weak-subsuper}.  Suppose that
\begin{equation}\label{eq:wc:data}
\underline v(0,\cdot)\leq\overline v(0,\cdot) \quad\text{for a.e. }r\in(0,1), \qquad \underline v(t,1)\leq\overline v(t,1) \quad\text{for a.e. }t\in(0,S).
\end{equation}
Then, one has the following comparison,
\begin{equation}\label{eq:wc:con}
\underline v(t,r)\leq\overline v(t,r) \quad\text{for a.e. }(t,r)\in(0,S)\times(0,1).
\end{equation}
If both functions are continuous, the conclusion holds pointwise for every $(t,r)\in[0,S]\times[0,1]$.
\end{proposition}

\begin{proof}
Set $d:=\underline v-\overline v$ and $z:=d_+$.  The boundary ordering in \eqref{eq:wc:data} gives $z\in L^2(0,S;\mathcal V_0)$. We first justify the positive-part test in the presence of nonhomogeneous boundary traces.  For $h\in(0,S)$, define the forward Steklov average
\begin{equation*}
d^h(t):=\frac{1}h\int_t^{t+h}d(s)\dd s, \quad \forall t\in(0,S-h).
\end{equation*}
Since $d(t,1)\leq0$ for a.e. $t$, one also has $d^h(t,1)\leq0$, and hence $(d^h)_+\in\mathcal V_0$.  Testing the time-averaged difference inequality by smooth Lipschitz approximations of $(d^h)_+$ and then letting $h\to0^+$ gives
\begin{equation}\label{eq:z:chai}
\int_0^t\langle d_t,z\rangle_{\mathcal V_0',\mathcal V_0}\dd s =\frac{1}{2}\|z(t)\|_{\mathcal H}^2-\frac{1}{2}\|z(0)\|_{\mathcal H}^2
\end{equation}
for a.e. $t\in(0,S)$.  In particular, no time regularity of the individual boundary traces beyond \eqref{eq:wc:data} is used.

Set $M:= \max\left\{
\|\underline v\|_{L^\infty((0,S)\times(0,1))},
\|\overline v\|_{L^\infty((0,S)\times(0,1))} \right\}$. 
Indeed, using the same nonnegative test function $\phi$ in the two variational inequalities and subtracting, we obtain
\begin{equation*}
\int_0^t\langle d_t,\phi\rangle_{\mathcal V_0',\mathcal V_0} \dd s +\int_0^t\int_0^1d_r\phi_r\,r^3\dd r\dd s
\leq \int_0^t\int_0^1 \bigl(F(r,\underline v)-F(r,\overline v)\bigr)\phi\,r^3\dd r\dd s.
\end{equation*}
The boundary ordering gives $z=d_+\in L^2(0,S;\mathcal V_0)$.  The preceding Steklov and truncation argument justifies taking $\phi=z$ and gives \eqref{eq:z:chai}.
Moreover, the Sobolev chain rule gives $d_rz_r=|z_r|^2$ almost everywhere.  We therefore obtain
\begin{equation}\label{eq:wc:ene}
\frac12\|z(t)\|_{\mathcal H}^2 +\int_0^t\|z_r(s)\|_{\mathcal H}^2\dd s
\leq \frac12\|z(0)\|_{\mathcal H}^2 +\int_0^t\int_0^1 \bigl(F(r,\underline v)-F(r,\overline v)\bigr)z\,r^3\dd r\dd s.
\end{equation}

On the set $\{z>0\}$, one has $z=\underline v-\overline v$, and, by \eqref{eq:rv:react},
\begin{align*}
F(r,\underline v)-F(r,\overline v)
=\int_{\overline v}^{\underline v}\partial_\zeta F(r,\zeta)\dd\zeta \leq 2M^2(\underline v-\overline v) =2M^2z.
\end{align*}
On the complementary set $\{z=0\}$, the product with $z$ vanishes.  Hence
\begin{equation*}
\bigl(F(r,\underline v)-F(r,\overline v)\bigr)z \leq2M^2z^2
\end{equation*}
almost everywhere.  Substituting this estimate into \eqref{eq:wc:ene}, we obtain
\begin{equation*}
\frac12\|z(t)\|_{\mathcal H}^2 \leq \frac12\|z(0)\|_{\mathcal H}^2 +2M^2\int_0^t\|z(s)\|_{\mathcal H}^2\dd s.
\end{equation*}

Since $z(0, \cdot)=0$ in $\mathcal{H}$, Gronwall's inequality applied to the preceding inequality gives $z(t)=0$ in $\mathcal{H}$ for almost every $t$. Since $z\in C([0,S];\mathcal H)$, the same conclusion holds for every $t\in[0,S]$. Hence $z(t, r)= 0$ for almost every $r\in (0, 1)$.  This proves \eqref{eq:wc:con}. If both functions are continuous, the almost-everywhere ordering extends to the closed cylinder by continuity.
\end{proof}

We finally record a criterion for converting piecewise differential inequalities and fixed-interface derivative jumps into weak subsolution and supersolution inequalities.  In the annular-barrier application, the inner piece extends smoothly across the center.  The origin-removability argument in the proof is included to cover energy-class profiles that are classical only for $r>0$, in particular the fractional CDY lower profile.

\begin{lemma}[Fixed-interface criterion]
\label{lem:fix-int}
Let $S>0$. Let $v\in\mathcal W_S\cap C([0,S]\times[0,1])$ and $ v_t\in L^2(0,S;\mathcal V_0')$.
Let $0<R_1<\cdots<R_N<1$ be a finite, possibly empty, set of fixed interfaces.  Suppose that $v$ is classical on each component of $(0,S)\times\bigl((0,1)\setminus\{R_1,\ldots,R_N\}\bigr)$.
Define the piecewise residual by
\begin{equation*}
G:=v_t-v_{rr}-\frac3r v_r-F(r,v),
\end{equation*}
and suppose that $G\in L^2(0,S;\mathcal H)$.
For each interface $R_j$, assume that the one-sided traces $v_r(t,R_j-)$ and $v_r(t,R_j+)$ exist for a.e. $t\in(0,S)$ and that both functions satisfy $R_j^3v_r(\cdot,R_j\pm)$ belong to $L^2(0,S)$. Recall the jump convention $[h]_{R}$  introduced in \eqref{eq:st:jumcon}. If
\begin{equation}\label{eq:fi:sign}
G\leq0 \quad\text{for a.e. }(t,r)\in(0,S)\times(0,1),
\quad [v_r]_{R_j}\geq0 \quad\text{for a.e. }t\in(0,S)
\end{equation}
for every $j$, then $v$ is a weak subsolution.  If both inequalities in \eqref{eq:fi:sign} are reversed, then $v$ is a weak supersolution.
\end{lemma}

\begin{proof}
We first take a nonnegative smooth test function $\phi$ such that $\phi(t,1)=0$ and $\phi$ vanishes near $r=0$.  Integrating by parts on each regular piece and summing over the pieces, we obtain
\begin{align}
&\int_0^S\langle v_t,\phi\rangle_{\mathcal V_0',\mathcal V_0}\dd t +\int_0^S\int_0^1v_r\phi_r r^3\dd r\dd t -\int_0^S\int_0^1F(r,v)\phi r^3\dd r\dd t \notag\\
&\qquad =\int_0^S\int_0^1G\phi r^3\dd r\dd t -\sum_{j=1}^N\int_0^S[v_r]_{R_j}(t)R_j^3\phi(t,R_j)\dd t.
\label{eq:fi:iden}
\end{align}
Thus, the right-hand side is nonpositive under \eqref{eq:fi:sign} and nonnegative when both inequalities are reversed.

It remains to remove the restriction that the test function vanish near the origin.  Let $\chi_\eta$ be a radial cutoff such that $0\leq\chi_\eta\leq1$, $\chi_\eta=0$ on $[0,\eta]$, $\chi_\eta=1$ on $[2\eta,1]$, and $|\chi_\eta'|\leq C/\eta$.  The four-dimensional Hardy inequality gives
\begin{equation}\label{eq:fi:hardy}
\int_0^1|\phi(r)|^2r\dd r \leq C\int_0^1|\phi_r(r)|^2r^3\dd r, \quad \forall\phi\in\mathcal V_0.
\end{equation}
In particular,
\begin{equation*}
\int_\eta^{2\eta}|\chi_\eta'(r)\phi(r)|^2r^3\dd r \leq C\int_\eta^{2\eta}|\phi(r)|^2r\dd r \longrightarrow0 \quad\text{as }\eta\to 0^+.
\end{equation*}
Together with dominated convergence for the remaining terms, this proves
$\chi_\eta\phi\longrightarrow\phi$ in $\mathcal V_0$.
Estimate \eqref{eq:fi:hardy} also shows that multiplication by $\chi_\eta$ is uniformly bounded on $\mathcal V_0$.  Hence the same convergence holds in $L^2(0,S;\mathcal V_0)$ for time-dependent test functions.  Since the cutoff preserves nonnegativity, nonnegative smooth test functions vanishing near $r=0$ are dense in the set of nonnegative functions in $L^2(0, S;\mathcal V_0)$.

Finally, all terms in the weak formulation are continuous on $L^2(0, S;\mathcal V_0)$.  Indeed, this follows from the assumptions on $v_t$ and $G$, the interface trace theorem, and the boundedness of $v$, which implies $F(r,v)\in L^2(0, S;\mathcal H)$.  We may therefore pass to the limit in \eqref{eq:fi:iden}.  This gives the required weak inequality for every nonnegative $\phi\in L^2(0, S;\mathcal V_0)$ and proves the lemma.
\end{proof}

\section{The CDY lower profile}\label{sec:cdy}

In this appendix, we verify the CDY lower profile from Proposition~\ref{lem:cdy-pro}. The pointwise calculation for $\underline q$ is classical and is included for readers' convenience;  see \cite{Chang-Ding-Ye-1992, Lin-Wang-2008}. We also verify the regularity needed to treat $\underline v$ as a weak subsolution for weak comparison in Appendix~\ref{subsec:weak-comp}.

\begin{proposition}\label{lem:cdy-pro}
Fix $\varepsilon\in(0,1)$ and set $a:=1+\varepsilon$.  There exist $\mu_0>1$ and $\kappa>0$, depending only on $\varepsilon$, with the following property.  Let $\mu\geq\mu_0$, and let $\ell:[0,T_*)\to(0,+\infty)$ be a $C^1$ nonincreasing function satisfying
\begin{equation}\label{eq:cdy:scasp}
0\leq-\ell'(t)\leq\frac{\kappa}{\mu}\ell(t)^\varepsilon.
\end{equation}
We further define
\begin{equation}\label{eq:cdy:sub}
\underline q(t,r):=Q_{\ell(t)}(r)+Q_\mu(r^{1+\varepsilon}).
\end{equation}
Then, $\underline q$ is a subsolution of \eqref{eq:flow:radial} in the sense of
\begin{equation}\label{eq:cdy:subine}
\underline q_t-\cL\underline q\leq0, \quad \forall (t,r)\in(0,T_*)\times(0,1).
\end{equation}

Moreover, for every $S\in(0,T_*)$, the function $\underline v:=\underline q/r$ belongs to $\mathcal W_S$, $\underline v_t\in L^2(0,S;\mathcal V_0')$, and has regular residual  $G\in L^2(0,S;\mathcal H)$ (see Lemma~\ref{lem:fix-int}).  Consequently, $\underline v$ is a weak subsolution in the sense of Definition~\ref{def:weak-subsuper}.
\end{proposition}

\begin{proof}[Proof of Proposition~\ref{lem:cdy-pro}]
Set $H(r):=Q_\mu(r^a)$.  Since $Q_\ell$ is a stationary harmonic map, a direct calculation gives
\begin{equation}\label{eq:cdy:tenid}
2r^2\cL(Q_\ell+H) =(a^2-\cos(2Q_\ell))\sin(2H) +\sin(2Q_\ell)(1-\cos(2H)).
\end{equation}
The second term can have either sign.  However, $(1-\cos(2H))/\sin(2H)=\tan H$, and $0\leq H(r)\leq2\arctan(1/\mu)$.  We choose $\mu_0$ so large that
\begin{equation*}
\tan H(r)\leq\frac{a^2-1}{2}, \quad \forall r\in[0,1], \; \forall \mu\in[\mu_0,+\infty).
\end{equation*}
Using $|\sin(2Q_\ell)|\leq1$ and $a^2-\cos(2Q_\ell)\geq a^2-1$, identity \eqref{eq:cdy:tenid} yields
\begin{equation}\label{eq:cdy:tenlow}
\cL(Q_\ell+H) \geq C_\varepsilon\frac{r^{\varepsilon-1}}\mu, \quad \forall r\in(0,1],
\end{equation}
after increasing $\mu_0$ once more.  Here we used $\sin(2H)\geq CH\geq Cr^a/\mu$ in the small range of $H$.

On the other hand,
\begin{equation*}
\underline q_t=-\frac{2r\ell'}{\ell^2+r^2} \leq\frac{2\kappa}{\mu} \frac{\ell^\varepsilon r}{\ell^2+r^2}.
\end{equation*}
Writing $s=r/\ell$, we obtain
\begin{equation*}
\frac{\ell^\varepsilon r}{\ell^2+r^2} =
r^{\varepsilon-1}\frac{s^{2-\varepsilon}}{1+s^2} \leq C_\varepsilon r^{\varepsilon-1}.
\end{equation*}
Selecting $\kappa>0$ small and combining this estimate with \eqref{eq:cdy:tenlow}, we obtain \eqref{eq:cdy:subine}. 

Finally, on every time interval on which $\ell$ is bounded away from zero, $\underline q/r=2/\ell+O(r^\varepsilon)$ and $\partial_r(\underline q/r)=O(r^{\varepsilon-1})$ near the origin, uniformly in time.  Its time derivative is bounded there.  
Moreover, identity \eqref{eq:cdy:tenid}, the bound $H=O(r^{1+\varepsilon})$, and \eqref{eq:rv:trans} show that the regular residual in the lifted equation is $O(r^{\varepsilon-2})$.  It therefore belongs to $L^2(0,S;\mathcal H)$, because $r^{2\varepsilon-4}r^3$ is integrable at the origin.   Thus, $\underline q/r$ belongs to $\mathcal W_S$, and its time derivative
and regular residual have the stated regularity. Lemma~\ref{lem:fix-int} therefore extends the pointwise inequality across the origin in the weak sense. This finishes the proof.
\end{proof}

For later reference, suppose that $\ell(t)\to0^+$ at a finite time $T_*$.  If a solution $q$ remains above $\underline q$, then evaluation at $r=\ell(t)$ gives
\begin{equation*}
q(t,\ell(t))\geq Q_{\ell(t)}(\ell(t))=\frac\pi2.
\end{equation*}
Since $q(t,0)=0$, the mean value theorem yields
\begin{equation}\label{eq:cdy:gralow}
\|q_r(t,\cdot)\|_{L^\infty(0,1)} \geq\frac\pi{2\ell(t)}.
\end{equation}
Thus, the collapse of the scale forces gradient blow-up.

\vspace{2mm}

\noindent\textbf{Acknowledgment.} The author is grateful to Jean-Michel Coron for many stimulating discussions. He is partially  supported by NSFC 12571474 and Shanghai Qiguang Natural Science Development Foundation. 
\vspace{2mm}

\noindent\textbf{Use of AI declaration.}
OpenAI’s GPT-5.6-sol model was used as an aid in mathematical discussions.  The author conceived and developed the work and finalized the manuscript.

\bibliographystyle{abbrv}
\bibliography{references}

\end{document}